\documentclass{scrartcl}
\usepackage[dvipsnames]{xcolor}
\usepackage{amsmath,amssymb,amsthm}
\usepackage{url}
\usepackage{ifdraft}

\ifdraft
{
  \usepackage{showkeys}
}
{
}

\usepackage{hyperref}
\usepackage{cleveref}
\usepackage{tikz}
\usetikzlibrary{arrows}
\usetikzlibrary{patterns,matrix,positioning}
\usetikzlibrary{arrows.meta, positioning, calc, decorations.pathreplacing, fit}
\usepackage{subcaption}  
\usepackage{tikz-cd}
\usepackage{float}
\usepackage{shuffle}
\usepackage{bbold} 

\usepackage{mathbbol}
\DeclareMathSymbol{\COMP}{\mathrel}{bbold}{\lq\;}
\DeclareMathSymbol{\rawCOMPDC}{\mathrel}{bbold}{\lq\;}

\usepackage[defernumbers=true,backend=biber,isbn=false,url=false,maxbibnames=999,maxcitenames=2,doi=false,giveninits=true,sortcites=true,style=alphabetic,eprint=false,hyperref=true]{biblatex}
\usepackage[normalem]{ulem}

\makeatletter
\renewcommand\@makefnmark{\hbox{\@textsuperscript{\normalfont\color{green}\@thefnmark}}}
\makeatother

\newcommand\REF[1]{\references{#1}}

\usepackage[linecolor=white,backgroundcolor=white,bordercolor=white,textsize=tiny,textcolor=magenta]{todonotes}

\newcommand\magenta[1]{{\color{magenta}{#1}}}
\newcommand\cyan[1]{{\color{cyan}{#1}}}

\newcommand\eps\epsilon

\newtheorem{theorem}{Theorem}[section]
\theoremstyle{plain}

\newtheorem{example}[theorem]{Example}

\newtheorem{notation}[theorem]{Notation}

\newtheorem{remark}[theorem]{Remark}

\newtheorem*{remark*}{Remark}
\newtheorem*{example*}{Example}

\theoremstyle{definition} 

\newtheorem{definition}[theorem]{Definition}
\newtheorem*{tata}{Generalization}
  {\begin{mdframed}[backgroundcolor=lightgray]\begin{tata}}%
  {\end{tata}\end{mdframed}}

\newcommand\K{\mathbb{K}}
\newcommand\R{\mathbb{R}}

\newcommand\BB{{\mathcal{B}}}
\newcommand\CC{{\mathcal{C}}} 
\newcommand\DD{{\mathcal{D}}}

\newcommand\DCC{{\mathbb{C}}} 
\newcommand\DCD{{\mathbb{D}}}

\newcommand\DGG{{\mathbb{G}}} 
\newcommand\DGH{{\mathbb{H}}}

\newcommand\D{{\mathcal{D}}}
\newcommand\N{\mathbb{N}}
\newcommand\Z{\mathbb{Z}}

\newcommand\id{\mathsf{id}}
\newcommand\DEF[1]{\textbf{#1}}

\renewcommand\O {\mathcal{O}}

\newcommand\GL{\mathsf{GL}}

\newif\ifshow
\showtrue
\showfalse

\definecolor{my-yellow}{HTML}{f0fff9} 

\ifshow
\newcommand\NEXTTIME[1]{\todo{NEXT TIME: #1}}
\usepackage{showkeys}
\else
\newcommand\NEXTTIME[1]{}
\fi

\newcommand\RECT{{\mathbb{R}\mathsf{ect}}}

\newcommand\leftboundary[1]{%
  \mathchoice
    {\uparrow\kern-0.4em\scalebox{0.4}{#1}} 
    {\uparrow\kern-0.4em\scalebox{0.4}{#1}} 
    {\uparrow\kern-0.0em\scalebox{0.4}{#1}} 
    {\uparrow\kern-0.1em\scalebox{0.4}{#1}} 
}
\newcommand\rightboundary[1]{
    \mathchoice
    {\scalebox{0.4}{#1}\kern-0.4em\uparrow}    
    {\scalebox{0.4}{#1}\kern-0.4em\uparrow}    
    {\scalebox{0.4}{#1}\kern-0.0em\uparrow}    
    {\scalebox{0.4}{#1}\kern-0.1em\uparrow}    
}

\newcommand\feedbackGG{\tau}        
\newcommand\actionGG{\triangleright}       
\newcommand\Aut{\operatorname{Aut}}

\newcommand\CAT{\underline{\mathbf{Cat}}}

\newcommand\SET{\underline{\mathbf{Set}}}

\newcommand\obj{\mathsf{obj}}
\renewcommand\hom{\mathsf{hom}}

\newcommand\categoryConstellation[5]{
  \begin{tikzpicture}[->,>=latex,node distance=4cm,scale=1.2,auto,baseline={(C1.base)}]
    \node (C1) at (0,0) {\( #2 \)};
    \node (C0) at (2,0) {\( #1 \)};
    
    \draw[->] (C1) to[bend left] node[above] {\( #3 \)} (C0);
    \draw[->] (C1) to[bend right] node[below] {\( #4 \)} (C0);
    \draw[->] (C0) to node[midway, anchor=center, fill=white, inner sep=0.2pt] {\( #5 \)} (C1);
  \end{tikzpicture}
}

\newcommand\inlineCategoryConstellation[5]{
        \categoryConstellation{#1}{#2}{#3}{#4}{#5}
}
\usepackage{enumitem}
\setlist[enumerate]{label=\roman*.}

\newcommand\INT{\underline{\mathbf{Int}}}
\newcommand\B{\mathbf{B}}
\renewcommand\i{\mathsf{i}}

\newcommand\Forget{\mathsf{Forget}}
\newcommand\Free{\mathsf{Free}}
\newcommand\QUIVER{\underline{\mathbf{Quiver}}}

\newcommand\FCQ{\Forget} 

\newcommand\FQC{\Free} 
\renewcommand\k{{\mathbf{k}}}

\newcommand{\SQUAREPATTERN}[5]{%
\scalebox{0.7}{$
  \begin{matrix}
   & #3 & \\
  #4 & #5 & #2 \\
   & #1 & \\
  \end{matrix}
$}
}

\definecolor{chartjs-red}{RGB}{255, 99, 132}
\definecolor{chartjs-orange}{RGB}{255, 159, 64}
\definecolor{chartjs-yellow}{RGB}{255, 205, 86}
\definecolor{chartjs-green}{RGB}{75, 192, 192}
\definecolor{chartjs-blue}{RGB}{54, 162, 235}
\definecolor{chartjs-purple}{RGB}{153, 102, 255}
\definecolor{chartjs-grey}{RGB}{201, 203, 207}
\definecolor{aggregation-color}{RGB}{255, 0, 255}
\newcommand\ac[1]{\textcolor{aggregation-color}{#1}}
\definecolor{new-green}{HTML}{69ff91}

\definecolor{id-arrow-color}{HTML}{03befc}

\usepackage{tcolorbox}
\tcbuselibrary{skins}

\definecolor{lightgreen}{RGB}{210,255,170}
\definecolor{lightpink}{RGB}{255,180,255}

\newtcbox{\greenbox}{
  enhanced,
  nobeforeafter,
  tcbox raise base,
  boxrule=0pt,
  interior style={lightgreen},
  frame hidden,
  colback=lightgreen,
  sharp corners,
}

\newtcbox{\pinkbox}{
  enhanced,
  nobeforeafter,
  tcbox raise base,
  boxrule=0pt,
  interior style={lightpink},
  frame hidden,
  colback=lightpink,
  sharp corners,
}

\makeatletter
\newcommand*{\relrelbarsep}{.386ex}
\newcommand*{\relrelbar}{%
  \mathrel{%
    \mathpalette\@relrelbar\relrelbarsep
  }%
}
\newcommand*{\@relrelbar}[2]{%
  \raise#2\hbox to 0pt{$\m@th#1\relbar$\hss}%
  \lower#2\hbox{$\m@th#1\relbar$}%
}
\providecommand*{\rightrightarrowsfill@}{%
  \arrowfill@\relrelbar\relrelbar\rightrightarrows
}
\providecommand*{\leftleftarrowsfill@}{%
  \arrowfill@\leftleftarrows\relrelbar\relrelbar
}
\providecommand*{\xrightrightarrows}[2][]{%
  \ext@arrow 0359\rightrightarrowsfill@{#1}{#2}%
}
\providecommand*{\xleftleftarrows}[2][]{%
  \ext@arrow 3095\leftleftarrowsfill@{#1}{#2}%
}
\makeatother

\newcommand\iC{\mathsf{i}} 
\newcommand\sC{\partial^-} 
\newcommand\tC{\partial^+} 

\begin{document}

\parindent0pt
\parskip1ex

\title{Aggregating time-series and image data: functors and double functors}

\author{Joscha Diehl\\
University of Greifswald}

\maketitle

\begin{abstract}

    Aggregation of time-series or image data over subsets of the domain
    is a fundamental task in data science.
    We show that many known aggregation operations can be interpreted as
    (double) functors on appropriate (double) categories.
    Such functorial aggregations are amenable to parallel implementation
    via straightforward extensions of Blelloch's parallel scan algorithm.
    In addition to providing a unified viewpoint on existing operations,
    it allows us to propose new aggregation operations for time-series
    and image data.

\end{abstract}

\tableofcontents

\section{Introduction}

After introducing basic concepts from category theory,
in \Cref{ss:categories_functors}
we formulate a large class of aggregation operations 
as functors on a certain category of intervals $\INT$, \Cref{ss:aggregation_as_functor}.
An essential ingredient is a certain, well-known, freeness
of $\INT$, \Cref{thm:free_category_over_quiver}.
Although this viewpoint might be folklore,
we observe that all aggregation functors
currently being used correspond to single-object
target categories. We sketch a possible relaxation in
\Cref{ex:multi_object_target}.
We recall Blelloch's prefix scan in \Cref{ss:parallel_scan_1}
and observe that it is easily useable for general
aggregation functors.

We recall basic concepts from the theory
of double categories in \Cref{ss:double_categories_double_functors}.
We introduce a class of aggregation
operators for planar data,
namely double functors on a certain category
of rectangles, \Cref{ss:aggregation_as_double_functor}
This category of rectangles is free in a certain sense,
\Cref{thm:free_double_category}.
Blelloch's scan is applicable ``slice-wise''
to these double functors, \Cref{ss:parallel_scan_2}.

\section{One parameter: sequential data}


Let sequential data $x_0, x_1, \ldots, x_{n-1}$ be given,
for example real-valued $x_k \in \R$.
For reasons that will become clear soon, we think
of this data as attached to the \emph{intervals}
$[0,1],[1,2],\ldots,[n-1,n]$.%
\footnote{
To be more precise, the data is attached to the \emph{boundaries}
of intervals and data attached to the intervals themselves
is obtained by some expression in the data attached to the boundaries.}

A common task in data science is to
``aggregate'' or ``summarize''
such (local) data on a larger intervals,
say $[0,n]$.
In the case of real-valued data, simple examples are
the sum and the max
\begin{align*}
    \sum_{k=0}^{n-1} x_k, \qquad \max_{k=0,\dots,n-1} x_k.
\end{align*}
Though very simple, these operations are of great importance
\cite{zaheer2017deep}. 

Less trivial examples are \emph{affine state-space models},
which we illustrate in a simplified version
of the case considered in
Mamba \cite{gu2023mamba}.
\begin{align*}
    y_0     &= 1 \text{ (some given value, chosen as $1$ for simplicity),} \\
    y_{k+1} &= y_k + x_k y_k, \quad k \in [0,n-1].
\end{align*}
The aggregated value $y_n$ is then
\begin{align*}
    y_n = 1 + \sum_{k\ge1} \sum_{0 \le i_1 < \dots < i_k < n} x_{i_1} \cdots x_{i_k}.
\end{align*}

People acquainted with ODEs or differential geometry
might recognize the discretization of a linear ODE,
or path development, here.
This hints at another aggregation.
First, interpolate $x$ to a continuous curve $X$,
piecewise affine between the values of $x$ (see \Cref{fig:interpolation}).
\begin{figure}[h]
    \centering
    \includegraphics[width=0.7\textwidth]{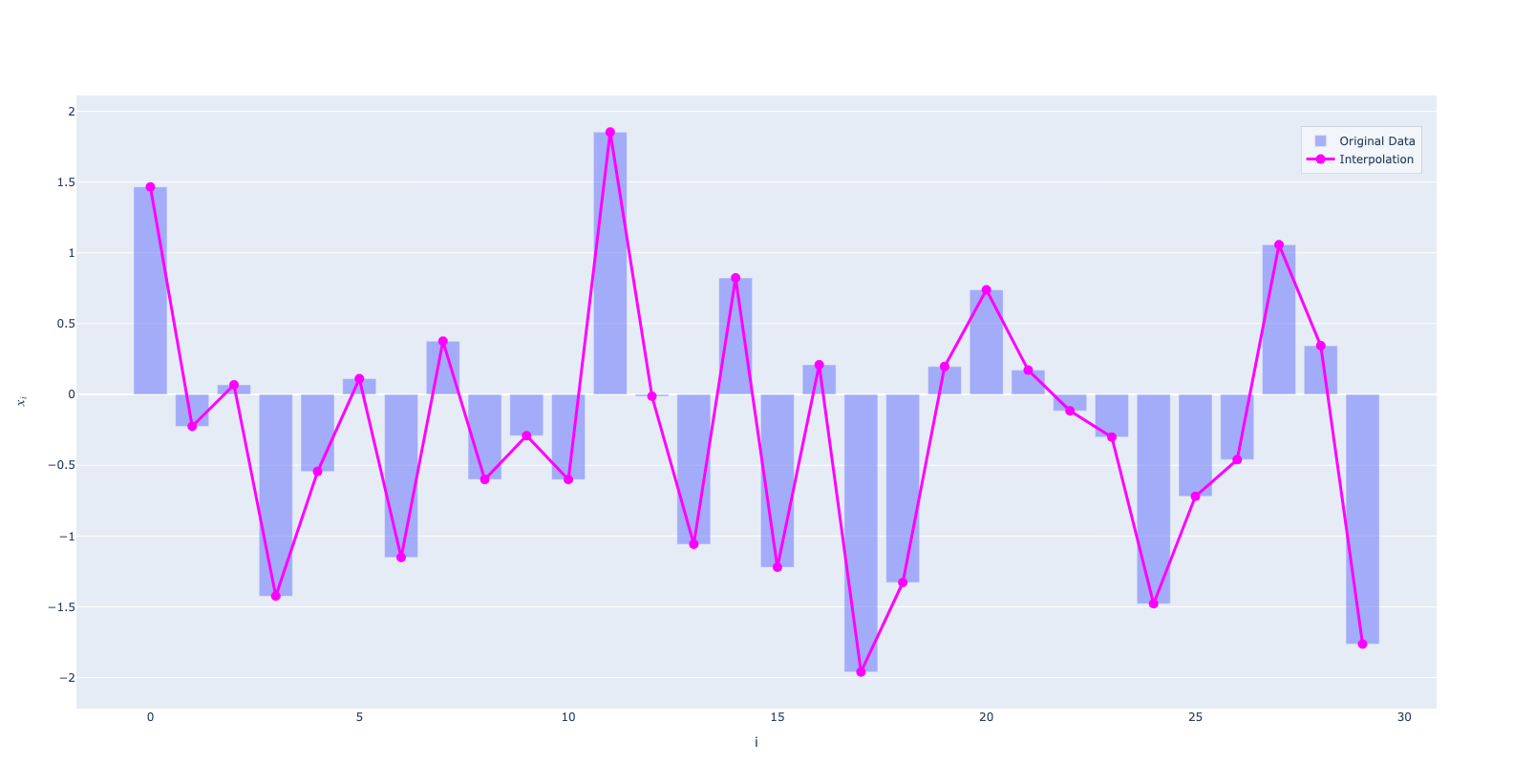}
    \caption{ 
        Interpolation of a time-series $x_0, x_1, \ldots$ to a continuous curve.
        The curve is piecewise affine between the values of $x$.
        }
    \label{fig:interpolation}
\end{figure}

Then, solve the ODE%
\footnote{
It turns out that this geometric approach is
only interesting if $x_0, \dots, x_n \in \R^d$ is multi-dimensional.
We then write $X = (X^{(1)}, \dots, X^{(d)})$ for the $d$-dimensional
interpolated curve.
We can of course ``lift'' a one-dimensional time-series to higher 
dimensions, by taking non-linear expressions in it.}
\begin{align}
    \label{eq:ODE}
    \begin{split}
    \dot Y_t  &= \sum_{k=1}^d A_k Y_t \dot X^{(k)}_t \\
         Y_0  &= \id_e.
    \end{split}
\end{align}
Here, $A_i \in \R^{e\times e}$ are square matrices
(parameters of the model) and $Y$ evolves
in the space of matrices.

\renewcommand\REF{\magenta{ \textbf{REF} }}
The aggregated value over $[0,n]$ is then taken to be the terminal value $Y_n \in \R^{e\times e}$.
It turns out
(see for example \cite{magnus1954exponential},\cite[Section 7.4]{friz2010multidimensional}),
there is a simple description
of the solution, where we ``stay on 
the discrete grid'' and do not have to interpolate explictly,
namely
\begin{align*}
    Y_n
    =
    \exp\left( \sum_{k=1}^d A_k (x^{(k)}_{n-1} - x^{(k)}_{n-2}) \right)
    \dots
    \exp\left( \sum_{k=1}^d A_k (x^{(k)}_1 - x^{(k)}_{0}) \right).
\end{align*}

All the aggregations presented thus far are then of the
following abstract form (see \Cref{fig:aggregation} for a visualization):
\begin{enumerate}

    \item Fix a semigroup $(\mathcal A, \bullet)$,
          i.e. a set $\mathcal A$ is endowed with an associative product $\bullet$.

    \item Map the real-valued time-series to a sequence of elements of $a_0, \dots, a_n \in \mathcal A$.

    \item Aggregate to the value
    \begin{align*}
        a_0 \bullet \dots \bullet a_n \in \mathcal A.
    \end{align*}

\end{enumerate}

\begin{figure}
    \centering
    \begin{tikzpicture}[
        node distance=2.7cm,
        line width=1.2pt,
        font=\footnotesize,
        arrows={-Latex} 
      ]
        \node[circle,fill=gray!10,inner sep=1pt] (0) at (0,0) {0};
        \node[circle,fill=gray!10,inner sep=1pt] (1) at (2,0) {1};
        \node[circle,fill=gray!10,inner sep=1pt] (2) at (4,0) {2};
        \node[circle,fill=gray!10,inner sep=1pt] (3) at (6,0) {3};
        \node[circle,fill=gray!10,inner sep=1pt] (4) at (8,0) {4};
        \draw[chartjs-blue] (0) -- (1);
        \draw[chartjs-blue] (1) -- (2);
        \draw[chartjs-blue] (2) -- (3);
        \draw[chartjs-blue] (3) -- (4);
        \draw[chartjs-blue,dotted] (4) -- ($(4)+(1,0)$);
        
        \node (a0) at ($(0)!0.5!(1)+(0,1)$) {$a_0$};
        \node (a1) at ($(1)!0.5!(2)+(0,1)$) {$a_1$};
        \node (a2) at ($(2)!0.5!(3)+(0,1)$) {$a_2$};
        \node (a3) at ($(3)!0.5!(4)+(0,1)$) {$a_3$};
        
        \draw[|->, dotted] ($(0)!0.5!(1)+(0,0.2)$) -- (a0);
        \draw[|->, dotted] ($(1)!0.5!(2)+(0,0.2)$) -- (a1);
        \draw[|->, dotted] ($(2)!0.5!(3)+(0,0.2)$) -- (a2);
        \draw[|->, dotted] ($(3)!0.5!(4)+(0,0.2)$) -- (a3);
    
        \draw[decorate, decoration={brace, amplitude=10pt}, line width=1.5pt, -{}] 
            ($(a0)+(-.5,.3)$) -- ($(a3)+(.5,.3)$) 
            node[midway, above=8pt] {$a_0 \bullet a_1 \bullet a_2 \bullet a_3$};
    \end{tikzpicture}
    \caption{
        Aggregation of a time-series $x_0, x_1, \ldots, x_{n-1}$ to $a_0 \bullet a_1 \bullet a_2 \bullet a_3$.
        The dotted arrows indicate the mapping from the time-series to the elements of the semigroup.
        }
    \label{fig:aggregation}
\end{figure}
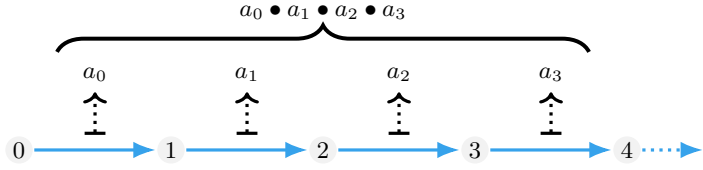

\begin{example}
\label{ex:all_aggregations}
The examples above fit in as follows:
\begin{enumerate}

    \item Take the semigroup $(\R,+)$ (which is actually a group)
            and take $a_i = x_i$. Then the aggregation yields
            \begin{align*}
                \sum_{i=0}^{n-1} x_i
            \end{align*}

    \item Take the semigroup $(\R,\max)$
            and take $a_i = x_i$. Then the aggregation yields
            \begin{align*}
                \max_{i=0,\dots,n-1} x_i
            \end{align*}

    \item
    \label{ex:all_aggregations:GL}
     Take the semigroup $(\GL_e(\R),\cdot)$ of invertible matrices
          with the usual matrix-product (this is actually a group)
          and take
          \begin{align*}
            a_0 &= \id_e \\
            a_. &= \exp\left( \sum_k A_k (x^{(k)}_i - x^{(k)}_{i-1}) \right).
          \end{align*}
          Then, the aggregation yields the terminal value
          of the solution to \eqref{eq:ODE}.

    \item Take the monoid $(\R,\cdot)$ of real numbers with multiplication,
          $a_i = 1 + x_i$.
          Then the aggregation yields
          \begin{align*}
                (1+x_0) \cdot \dots \cdot (1+x_{n-1}) = 1 + \sum_{k\ge1} \sum_{0 \le i_1 < \dots < i_k < n} x_{i_1} \cdots x_{i_k}.
          \end{align*}
                
\end{enumerate}

We mention two more example:
\newcommand\w[1]{\cyan{\mathtt{#1}}}
\begin{enumerate}

    \setcounter{enumi}{4}
    \item  

       Take
       \begin{align*}
          a_i := \sum_{n \ge 0} (x_i)^n [\w{1}^n],
       \end{align*}
       a formal sum, with real coefficients, 
       of the symbols $\mathcal A := \{ [\w{1}^n] \mid n \in \N_{\ge 0} \}$.
       Then, the aggregation
       \begin{align*}
        a_0 \otimes \dots \otimes a_n,
       \end{align*}
       calculated in $T((\mathcal A))$, the (completed) tensor algebra over $\mathcal A$,
       yields the \emph{iterated-sums signature} \cite{diehl2020time,DEFT22}.
       We give a few of the terms appearing:
       \begin{align*}
            a_0 \otimes \dots \otimes a_n
            &=
            1
            +
            \sum_{0 \le i \le n} x_i [\w{1}]
            +
            \sum_{0 \le i_1 < i_2 \le n} x_{i_1} x_{i_2} [\w{1}][\w{1}] \\
            &\qquad
            +
            \sum_{0 \le i \le n} (x_i)^2 [\w{1}^2]
            +
            \sum_{0 \le i_1 < i_2 \le n} (x_{i_1})^2 x_{i_2} [\w{1}^2][\w{1}]
            +
            \dots
       \end{align*}

    \item 
       Take
       \begin{align*}
          a_i := \exp( x_{i+1} - x_i ),
       \end{align*}
       as a formal exponential in the (completed) tensor algebra $T((\R))$.
       Then, the aggregation
       \begin{align*}
            a_0 \otimes \dots \otimes a_n
       \end{align*}
       yields the \emph{iterated-integrals signature} (\cite{chen1954iterated,lyons2007differential})
       of the piecewise-linearly interpolated curve corresponding to $x$
       (see \Cref{fig:interpolation}).%
       \footnote{
        We note that,
        different from the iterated-sums signature,
        the iterated-integrals signature
        is very degenerate for one-dimensional data,
        so the reader is invited to actually have $x_i \in \R^q$ in mind.
       }

\end{enumerate}
    
\end{example}

Associativity of the operation is important for at least two related reasons:
\begin{itemize}
    \item the aggregation should not depend on the order in which the products are evaluated, 
    \item the aggregation should be parallelizable (see \Cref{ss:parallel_scan_1}).
\end{itemize}

What turns out \emph{not} to be essential,
is the fact that the operation is defined on only \emph{one} set.
This leads us to the generalization of \Cref{ss:aggregation_as_functor}.

\subsection{Categories, functors and free categories}
\label{ss:categories_functors}

\begin{quote}
    \itshape
    {
    A \emph{category} is a set of ``objects'' and a set of ``morphisms'' between them
    such that the morphisms can be composed and that there are identity morphisms.
    A \emph{functor} is a structure-preserving map between categories.
    }
\end{quote}

A \DEF{category}%
\footnote{Recommended references for category theory are: \cite{mac2013categories,awodey2010category}.}
$\CC$ consists of
\begin{itemize}
  \item a set of \DEF{objects} $\CC_0$,
  \item a set of \DEF{morphisms} $\CC_1$ (also called \DEF{arrows})
  \item
    a map $\iC_\CC: \CC_0 \to \CC_1$ assigning to each object $X \in \CC_0$ an \DEF{identity morphism} $\iC_\CC(X) \in \CC_1$,

  \item \DEF{source} and \DEF{target} maps $\sC_\CC, \tC_\CC: \CC_1 \to \CC_0$,

  \item
    for all objects $X,Y,Z \in \CC_0$ and all $f,g \in \CC_1$
    with $\tC_\CC(f) = \sC_\CC(g)$ (we call such $f,g$ \DEF{composable}),
    a morphism
    $f \COMP_\CC g \in \CC_1$
    \footnote{Often written as $g \circ f$ or $gf$, but we prefer this ``diagrammatic'' notation.}
    the \DEF{composition} of $f$ and $g$,

\end{itemize}
such that the following hold
\begin{itemize}
\item 
for composable morphisms $f,g \in \CC_1$,
\begin{align*}
    \sC_\CC(f \COMP_\CC g) = \sC_\CC(f) \quad \text{ and } \tC_\CC(f \COMP_\CC g) = \tC_\CC(g),
\end{align*}
i.e. the source of composition is the source of the first map,
the target is the target of the second map,

\item
\DEF{associativity} holds for all composable morphisms $f,g,h \in \CC_1$
\begin{align*}
    f \COMP_\CC (g \COMP_\CC h) = (f \COMP_\CC g) \COMP_\CC h,
\end{align*}

\item
the \DEF{identity} morphisms behave as expected,
namely for any $f \in \CC_1$,
\begin{align*}
    \iC_\CC(\sC_\CC(f)) \COMP_\CC f = f, \quad f \COMP_\CC \iC_\CC(\tC_\CC(f)) = f.
\end{align*}

\end{itemize}

\begin{figure}
    \centering
    \begin{tikzpicture}[
        point/.style={circle, fill=black, inner sep=2pt},
        arrow/.style={-Stealth, thick},
        map/.style={->, dashed, thick},
        >=Stealth
    ]
        \node[align=center] at (0,2) {\textbf{Structures in a category $\mathcal{C}$}};
        
        \node[point] (X) at (-3,0) {};
        \node[point] (Y) at (0,0) {};
        \node[point] (Z) at (3,0) {};
        
        \draw[arrow] (X) to[bend left=40] node[midway, above] {$f$} (Y);
        \draw[arrow] (Y) to[bend left=40] node[midway, above] {$g$} (Z);
        \draw[arrow] (X) to[bend right=40] node[midway, below] {$f \COMP g$} (Z);
        
        \draw[arrow, id-arrow-color] (X) to[loop left] node[left] {\scalebox{0.7}{$\i_\mathcal{C}(X)$}} (X);
        \draw[arrow, id-arrow-color] (Y) to[loop left] node[left] {\scalebox{0.7}{$\i_\mathcal{C}(Y)$}} (Y);
        \draw[arrow, id-arrow-color] (Z) to[loop left] node[left] {\scalebox{0.7}{$\i_\mathcal{C}(Z)$}} (Z);

        \node[fill=white, fill opacity=0.7, inner sep=2pt, anchor=north] at (X.south) {\scalebox{0.7}{$X=\sC_\CC(f)$}};
        \node[fill=white, fill opacity=0.7, inner sep=2pt, anchor=north] at (Y.south) {\scalebox{0.7}{$\tC_\CC(f)=Y=\sC_\CC(g)$}};
        \node[fill=white, fill opacity=0.7, inner sep=2pt, anchor=north] at (Z.south) {\scalebox{0.7}{$Z=\tC_\CC(g)$}};
    \end{tikzpicture}
    \caption{Visualization of a category $\mathcal{C}$, showing objects (points), morphisms (arrows), 
    identity morphisms (blue loops), the source and target maps as well as composition.}
\end{figure}
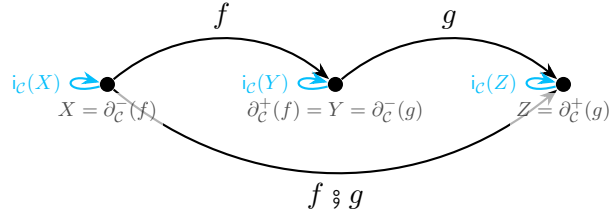

\begin{remark}~

  \begin{enumerate}

  \item 
    We speak of \emph{sets}, of objects for example, although in general,
    the objects and morphisms of a category are not sets. (They are often ``too large'' to be sets.)

  \item
    An alternative, and more common, definition of a category is to define
    it as a collection of objects $\obj(\CC)$ 
    and, for every $X,Y \in \obj(\CC)$ a collection of morphisms $\hom_\CC(X,Y)$
    with a distinguished idenity morphism $\id_X \in \hom_\CC(X,X)$ for every object $X$,
    as well as as an appropriate composition of morphisms.

    The translation between the two definitions is as follows:
    \begin{align*}
        \obj(\CC) &\leftrightarrow \CC_0 \\
        \hom_\CC(X,Y) &\leftrightarrow \{ f \in \CC_1 \mid \sC_\CC(f) = X, \tC_\CC(f) = Y \} \\
        \id_X &\leftrightarrow i_\CC(X).
    \end{align*}
    The definition we gave first generalizes more easily to 
    ``higher-dimensional categories'', \Cref{sec:two-param}.
    We will use both conventions interchangeably.

  \item
    We also write $f: X \to Y$ if $f \in \hom_\CC(X,Y)$, even though
    \begin{center}
      \underline{$X,Y$ are, in general, not sets and $f$ is not a map between sets}.
    \end{center}
  \end{enumerate}

\end{remark}

\begin{example}~
    \label{ex:categories}

    \begin{enumerate}

        \item
        We define the category $\INT$  ``of intervals with endpoints in $\Z$'' as follows:
        \footnote{For experts: this is just the poset $\Z$ considered as a category.}
        \begin{align*}
            \CC_0 = \Z, \quad \CC_1 = \{ [m,n] \mid m, n \in \Z, m \le n \}, \\
            \sC_{\CC}([m,n]) = m, \quad \tC_{\CC}([m,n]) = n, \\
            \iC_{\CC}(m) = [m,m], \quad [m,n] \COMP [n,p] = [m,p].
        \end{align*}
        See \Cref{fig:int-category} for a visualization.

        This category will form the \emph{domain} of our functors built from
        sequential data.

        \begin{figure}[h]
            \centering 
            \begin{tikzpicture}[
                scale=1.2,
                node distance=2cm,
                thick,
                point/.style={circle, draw, minimum size=0.6cm, inner sep=1pt, font=\normalsize},
                arrow/.style={-Stealth, thick},
                interval/.style={-Stealth, thick}
            ]
                \foreach \x in {0,...,4}{
                    \node[point] (p\x) at (\x*1.5,0) {$\x$};
                }
                
                \foreach \x in {0,...,4}{
                    \draw[arrow, id-arrow-color] (p\x) to[loop above] node[above, font=\small] {$[\x,\x]$} (p\x);
                }
                
                \draw[interval] (p0) to node[above, font=\small] {$[0,1]$} (p1);
                \draw[interval] (p1) to node[above, font=\small] {$[1,2]$} (p2);
                \draw[interval] (p2) to node[above, font=\small] {$[2,3]$} (p3);
                \draw[interval] (p3) to node[above, font=\small] {$[3,4]$} (p4);
                
                \draw[interval, bend right=20] (p0) to node[fill=white, fill opacity=0.7, font=\small] {$[0,2]$} (p2);
                \draw[interval, bend right=20] (p1) to node[fill=white, fill opacity=0.7, font=\small] {$[1,3]$} (p3);
                \draw[interval, bend right=20] (p2) to node[fill=white, fill opacity=0.7, font=\small] {$[2,4]$} (p4);
                
                \draw[interval, bend right=30] (p0) to node[fill=white, fill opacity=0.7, font=\small] {$[0,3]$} (p3);
                \draw[interval, bend right=30] (p1) to node[fill=white, fill opacity=0.7, font=\small] {$[1,4]$} (p4);
                
                \draw[interval, bend right=40] (p0) to node[fill=white, fill opacity=0.7, font=\small] {$[0,4]$} (p4);
                
            \end{tikzpicture}
            \caption{A subset of the category $\INT$ showing all intervals with endpoints in $\{0,1,2,3,4\}$.}
            \label{fig:int-category}
        \end{figure}
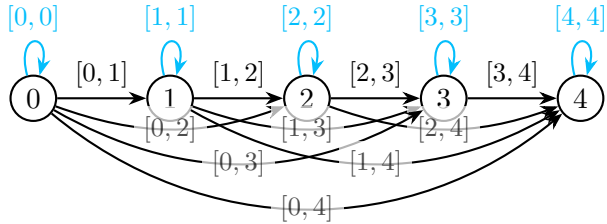

        \item
        \label{ex:categories:delooping}
        The \DEF{delooping} of a monoid, $M$
        \begin{align*}
            \obj(\B M)       &:= \{*\} \\
            \hom_{\B M}(*,*) &:= M,
        \end{align*}
        with composition given by multiplication in $M$, see \Cref{fig:delooping}.
        This category will usually form the \emph{codomain} of our functors built from
        sequential data.
        \begin{figure}[h]
            \centering 
            \begin{tikzpicture}[
                arrow/.style={-Stealth, thick},
                node distance=3cm,
                thick,
                main node/.style={circle, minimum size=1cm, font=\Large}
            ]
                \node[main node] (star) {$*$};
                
                \draw[->, out=150, in=110, looseness=5, id-arrow-color] (star) to node[pos=0.5, fill=white, fill opacity=0.7] {$e$} (star);
                \draw[->, out=150, in=110, looseness=9] (star) to node[above, pos=0.5] {$m$} (star);
                \draw[->, out=150, in=110, looseness=15] (star) to node[above, pos=0.5] {$m'$} (star);
            \end{tikzpicture}
            \caption{The delooping of a monoid $M$ is a category with one object $*$ and morphisms corresponding to the elements of $M$.}
            \label{fig:delooping}
        \end{figure}
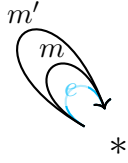

    \end{enumerate}
    
\end{example}


A (covariant) \DEF{functor} $F$ between categories $\CC, \DD$,
written $F: \CC \to \DD$ consists of two maps $F_0: \CC_0 \to \DD_0$ and $F_1: \CC_1 \to \DD_1$,
such that
\begin{itemize}
    \item 
    $\iC_\DD(F_0(X)) = F_1(\iC_\CC(X))$ for every object $X \in \CC_0$,
    
    \item 
    $\sC_\DD(F_1(f)) = F_0(\sC_\CC(f))$ and $\tC_\DD(F_1(f)) = F_0(\tC_\CC(f))$ for every morphism $f \in \CC_1$,

    \item
    $F_1( f \COMP_\CC g ) = F_1(f) \COMP_\DD F_1(g)$ for all composable morphisms $f,g \in \CC_1$.
\end{itemize}

\begin{notation}
    \label{not:functor}
    We sometimes write $F = (F_0,F_1)$ to emphasize the two components of a functor.
\end{notation}


\begin{figure}
    \centering

    \begin{tikzpicture}[
        point/.style={circle, fill=black, inner sep=2pt},
        arrow/.style={-Stealth, thick},
        functor/.style={-Stealth, thick, magenta, dashed, bend left=15},
        >=Stealth
    ]
        \begin{scope}[shift={(-4,0)}]
            \node at (0,3.5) {\textbf{Category $\mathcal{C}$}};
            
            \node[point, label=below:{$A$}] (A) at (-1,0) {};
            \node[point, label=below:{$B$}] (B) at (1,0) {};
            \node[point, label=above:{$C$}] (C) at (0,1.5) {};
            
            \draw[arrow] (A) to node[below] {$f$} (B);
            \draw[arrow] (A) to node[left] {$h$} (C);
            \draw[arrow] (B) to node[right] {$g$} (C);
            
            
            \node[align=center] at (0,-1.5) {$f \COMP g = h$ (composition)};
        \end{scope}
        
        \begin{scope}[shift={(4,0)}]
            \node at (0,3.5) {\textbf{Category $\mathcal{D}$}};
            
            \node[point, label=left:{$F(A)$}] (FA) at (-1.5,1) {};
            \node[point, label=right:{$F(B)$}] (FB) at (1.5,1) {};
            \node[point, label=below:{$F(C)$}] (FC) at (0,0) {};
            \node[point, label=above:{$D$}] (D) at (0,2) {};  
            
            \draw[arrow] (FA) to node[below] {$F(f)$} (FB);
            \draw[arrow] (FA) to node[left] {$F(h)$} (FC);
            \draw[arrow] (FB) to node[right] {$F(g)$} (FC);
            \draw[arrow] (FA) to node[left] {$j$} (D);
            \draw[arrow] (FB) to node[right] {$k$} (D);
            
            \node[align=center] at (0,-1.5) {$F(h) = F(f \COMP g) = F(f) \COMP F(g)$};
        \end{scope}
        
        \draw[functor] (-3,2.5) to node[midway, above] {Functor $F: \mathcal{C} \rightarrow \mathcal{D}$} (3,2.5);
    \end{tikzpicture}
    
    \caption{A functor $F$ between two categories $\mathcal{C}$ and $\mathcal{D}$. The functor preserves the structure of the categories, including composition of morphisms.
    For legibility, the identity morphisms are not shown.}
\end{figure}
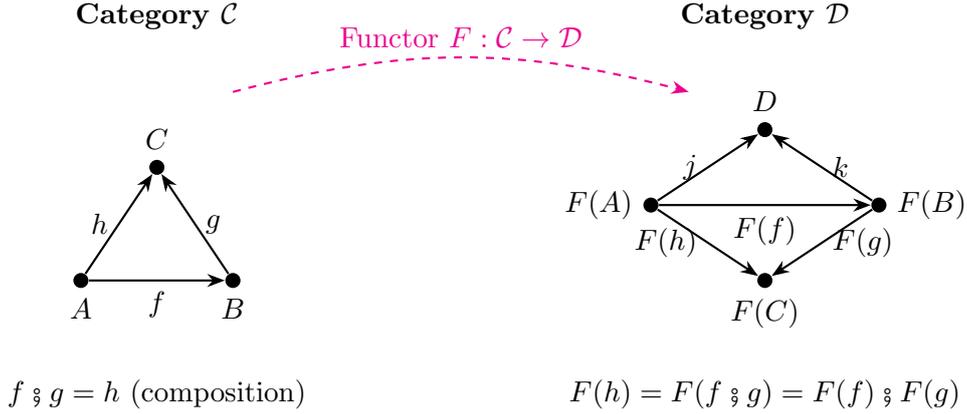

\begin{example}~
    \label{ex:functors}

    \begin{enumerate}
        \item
            Let $M,N$ be two monoids and $\B M, \B N$ their deloopings (\Cref{ex:categories},
                        \ref{ex:categories:delooping}).
            Then, functors $F: \B M\to \B N$ are in one-to-one correspondence
            to monoid homomorphisms $f: M \to N$.

        \item
            \label{ex:functors:sum}
            Let $x_i \in \R$, $i \in \Z$, be a time-series.
            Let $\B \R$ be the delooping of the additive group $\R$.
            Define
            \begin{align*}
                F( m )     &:= * \\
                F( [m,n] ) &:= \sum_{k=m}^{n-1} x_k, \quad m \le n.
            \end{align*}
            Then, $F$ is a functor $\INT \to \B \R$.

        \item
        \label{ex:functors:monoid}
        More generally, let $M$ be a monoid and $a_i \in M$, $i \in \Z$.
        (For example $M = \GL_e(\R)$ and $a_i = \exp\left( \sum_i A_i (x^{(i)}_i - x^{(i)}_{i-1}) \right)$.)
        Let $\B M$ be the delooping of $M$.
        Define
        \begin{align*}
            F( m )     &:= * \\
            F( [m,n] ) &:= a_m \bullet \dots \bullet a_{n-1}, \quad m \le n.
        \end{align*}
        This covers all examples in \Cref{ex:all_aggregations}.

    \end{enumerate}
\end{example}

The fact that \Cref{ex:functors}, \ref{ex:functors:monoid} is a functor
can be easily verified directly.
Instead, we prove that $\INT$ is a certain free category,
which gives us a simple way to describe \emph{all} functors out of $\INT$.

Recall from linear algebra that any linear map out of an $\R$-vector space $V$
(which is a \emph{free module} over $\R$)
is uniquely specified by
an (arbitrary) assignment of values
to a basis of $V$.
Analogously, a functor out of a \emph{free category}
will be uniquely determined by an assignment 
on a certain substructure of the category.
The relevant substructure for free modules are bases;
for free categories they are quivers.

A \DEF{quiver}\footnote{A ``directed graph with multi-edges'', also called $1$-polygraph \cite[Chapter 1]{ararewriting}.}
consists of a set of \DEF{vertices} $Q_0$ and a set of \DEF{arrows} $Q_1$,
together with two maps (\DEF{source} and \DEF{target})
\begin{align*}
    Q_1 \xrightrightarrows[\tC_Q]{\sC_Q} Q_0
\end{align*}

A \DEF{morphism} $f: Q \to Q'$ of quivers is a pair of maps $f_0: Q_0 \to Q'_0$ and $f_1: Q_1 \to Q'_1$
that are compatible with the source and target maps:
\begin{align*}
    f_1 \COMP \sC_{Q'} = \sC_Q \COMP f_0, \quad f_1 \COMP \tC_{Q'} = \tC_Q \COMP f_0.
\end{align*}
Their composition is defined componentwise.
This turns the class of quivers into a category $\QUIVER$.
Consider the ``forgetful'' functor
\begin{align*}
    \FCQ
    : \CAT \to \QUIVER,
\end{align*}
on the category of (small\footnote{Again, we do not dwell on cardinality issues.}) categories,
which forgets the identity map and the composition that a category possess.
%
The following result is well-known.
\begin{theorem}[Free category over a quiver%
        \footnote{\cite[Chapter 3]{higgins1971categories},
        \cite[p.49]{mac2013categories},
        \cite[2.6.16]{barr1990category},
        \cite[p.532]{street1996categorical},
        \cite[Lemma 2.1.3]{ararewriting}}
    ]
	\label{thm:free_category_over_quiver}
	Let
	\begin{align*}
        Q_1 \xrightrightarrows[\tC_Q]{\sC_Q} Q_0
	\end{align*}
	be a quiver.
	Then there exists a category $\FQC(Q)$
	and a map of quivers $\iota: Q \to \FCQ( \FQC(Q) )$
	which is \emph{free} on $Q$ in the following sense:
	for every category $\mathcal D$ and every map $f: Q \to \FCQ( \mathcal D)$
	of quivers,
    there exists a unique functor $F: \FQC(Q) \to \mathcal D$
	satisfying $f = \iota \COMP \FCQ(F)$,
    which, with abuse of notation (by omitting the $\FCQ$-functor), is visualized as follows
	\begin{center}
		\begin{tikzcd}
		Q \arrow[d, "\iota"] \arrow[r, "f"]    & \mathcal D \\
		\FQC(Q) \arrow[ru, "\exists ! F"', dotted] &           
		\end{tikzcd}
	\end{center}

	Moreover, $\FQC: \QUIVER \to \CAT$ is a functor.%
    \footnote{
        Moreover, the association $f \mapsto F$ is \emph{natural} in $Q$ and $\mathcal D$;
        which means that $\FCQ$ is \emph{left adjoint} to $\FQC$.}
\end{theorem}

\begin{example}



		
		Let $Q_1 = \{ *\rightarrow *\}, Q_0 = \{*\}$.
		Then $\FQC(Q)$ is
		the
        one-object category with the
        free monoid on one generator (i.e. $(\N_0,+)$)
        as its set of arrows.

	
\end{example}

\begin{proof}

	Let $P^{[n]}$ be the quiver with $n+1$ vertices $\{0,1,\dots,n\}$ and $n$ arrows $0 \to 1 \to \dots \to n$.
	A \textbf{directed path of length $n$ from $i$ to $j$} in a quiver $Q$ is a map of quivers $P^{[n]} \to Q$
	with $0\mapsto i, n \mapsto j$. We call $i$ the source and $j$ the target of the path.
	In particular, there is a unique path of length $0$ from $i$ to $i$ for each $i \in Q_0$.
    We then set
    \begin{align*}
        \FQC(Q)_0 &:= Q_0, \\
        \FQC(Q)_1 &:= \{ \text{directed paths in $Q$} \},
    \end{align*}
    where the source and target of a directed path are as above.
    Composition in this category is defined to be concatenation of paths,
    which is clearly associative. The zero-length directed paths are the identities.


    It remains to show the universal property.
    Let $f: Q \to \FCQ( \mathcal D)$ be a map of quivers.
    Define for $x \in \FQC(Q)_0 = Q_0$ 
    \begin{align*}
        F_0(x) := f_0(x).
    \end{align*}

    For $p = (i=i_0 \to i_1 \to \dots i_n = j)$ a directed path in $Q$ from $i$ to $j$, define
    \begin{align*}
        F_1(p)
        &:= F_1( i_0 \to i_1 ) \COMP_{\mathcal D} F_1( i_1 \to i_2 ) \COMP_{\mathcal D} \dots \COMP_{\mathcal D} F_1( i_{n-1} \to i_n ) \\
        &:= f_1( i_0 \to i_1 ) \COMP_{\mathcal D} f_1( i_1 \to i_2 ) \COMP_{\mathcal D} \dots \COMP_{\mathcal D} f_1( i_{n-1} \to i_n ).
    \end{align*}
    It is clear that this is well-defined and that $F_1$ is a functor.
    Moreover, by functoriality it is uniquely specified by its value on the generators
    (i.e. paths of length $1$) where it coincides with $f_1$.
    This gives uniqueness of $F$.
\end{proof}

\subsection{Aggregation as a functor on a category of intervals}
\label{ss:aggregation_as_functor}
\begin{quote}
    \itshape
    {
    Many existing aggregation operations can be interpreted as functors on a category of intervals.
    Moreover, new aggregation operations can be defined in this way.
    }
\end{quote}

\newcommand\EI{E}
    Consider the following quiver of ``elementary intervals with endpoints in $\Z$'':
    \begin{align*}
        \EI_0 &= \Z, \\
        \EI_1 &= \{ [n,n+1] \mid n \in \Z \}, \\
        \sC_\EI([n,n+1]) &= n, \quad \tC_\EI([n,n+1]) = n+1.
    \end{align*}

    Then, looking at the proof of \Cref{thm:free_category_over_quiver},
    we see that (a copy of) $\FQC(\EI)$ is given by $\INT$.
    
    \bigskip
    This gives us the following way to aggregate values attached
    to elementary intervals.
    Let $\D$ be any category,
    and let $f: \EI \to \D$ be a map of quivers.
    This means 
    \begin{align*}
        \sC_\D( f_1( [n,n+1] ) ) = f_0(n), \quad \tC_\D( f_1( [n,n+1] ) ) = f_0(n+1).
    \end{align*}
    In the case where $\D$ has only one object,
    for example when $\D = \B M$ is the delooping of a monoid $M$,
    then this condition is void.

    Then, there exists a unique ``lift'' of $f$ to a functor $F: \INT \to \D$.
    This means that there exist (unique) \underline{aggregated values}
    (see \Cref{fig:aggregation_in_a_category}
    for an example visualization))
    \begin{align*}
        F( [m,n] ) \in D_1, \qquad \forall m \le n.
    \end{align*}
    The special case of  a monoid $M$, $\D = \B M$ and $f([n,n+1]) = a_n \in M$
    results in the functor $F$ of \Cref{ex:functors}, \ref{ex:functors:monoid}.
    Although this result is easily obtainable ``by hand'' without
    the language of categories, it will extrapolate nicely to the
    two-parameter case, \Cref{sec:two-param}.

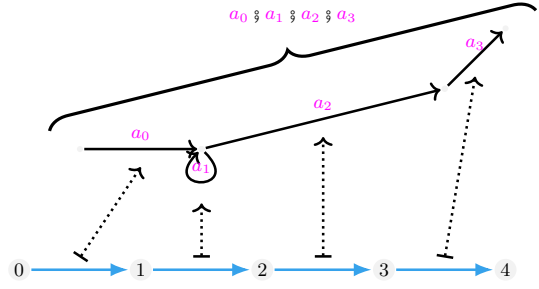
\begin{figure}
    \centering
    \scalebox{0.8}{
\begin{tikzpicture}[
    node distance=2.7cm,
    line width=1.2pt,
    font=\footnotesize,
    arrows={-Latex} 
  ]
    \node[circle,fill=gray!10,inner sep=1pt] (0) at (0,0) {0};
    \node[circle,fill=gray!10,inner sep=1pt] (1) at (2,0) {1};
    \node[circle,fill=gray!10,inner sep=1pt] (2) at (4,0) {2};
    \node[circle,fill=gray!10,inner sep=1pt] (3) at (6,0) {3};
    \node[circle,fill=gray!10,inner sep=1pt] (4) at (8,0) {4};
    \draw[chartjs-blue] (0) -- (1);
    \draw[chartjs-blue] (1) -- (2);
    \draw[chartjs-blue] (2) -- (3);
    \draw[chartjs-blue] (3) -- (4);
    
    \node[circle,fill=gray!10,inner sep=1pt] (a0start) at (1,2) {};
    \node[circle,fill=gray!10,inner sep=1pt] (a0end) at (3,2) {};
    \draw[->] (a0start) to node[above,aggregation-color] {$a_0$} (a0end);
    
    \node[circle,fill=gray!10,inner sep=1pt] (a1node) at (3,2) {};
    \draw[->] (a1node) to [out=-45,in=-135,looseness=24] node[above,aggregation-color] {$a_1$} (a1node);
    
    \node[circle,fill=gray!10,inner sep=1pt] (a2start) at (3,2) {};
    \node[circle,fill=gray!10,inner sep=1pt] (a2end) at (7,3) {};
    \draw[->] (a2start) to node[above,aggregation-color] {$a_2$} (a2end);
    
    \node[circle,fill=gray!10,inner sep=1pt] (a3start) at (7,3) {};
    \node[circle,fill=gray!10,inner sep=1pt] (a3end) at (8,4) {};
    \draw[->] (a3start) to node[above,aggregation-color] {$a_3$} (a3end);
    
    \draw[|->, dotted] ($(0)!0.5!(1)+(0,0.2)$) -- ($(a0start)!0.5!(a0end)+(0,-0.3)$);
    \draw[|->, dotted] ($(1)!0.5!(2)+(0,0.2)$) -- ($(a1node)+(-0.0,-0.9)$);
    \draw[|->, dotted] ($(2)!0.5!(3)+(0,0.2)$) -- ($(a2start)!0.5!(a2end)+(0,-0.3)$);
    \draw[|->, dotted] ($(3)!0.5!(4)+(0,0.2)$) -- ($(a3start)!0.5!(a3end)+(0,-0.3)$);

    \draw[decorate, decoration={brace, amplitude=10pt}, line width=1.5pt, -{}] 
        ($(a0start)+(-.5,.3)$) -- ($(a3end)+(.5,.3)$) 
        node[midway, above=18pt] {$\ac{a_0} \COMP \ac{a_1} \COMP \ac{a_2} \COMP \ac{a_3}$};
\end{tikzpicture}}
\caption{ Aggregation in a category }
\label{fig:aggregation_in_a_category}

\end{figure}

In applications to time-series data $x_i, i \in \Z$,
the values $f_1([n,n+1])$ will usually be obtained
as an expression in the values $x_n$ and $x_{n+1}$.
For example, to replicate
\ref{ex:all_aggregations},
\ref{ex:all_aggregations:GL}
we take $\DD$ to be the delooping of $\GL_e(\R)$
and $f_1([n,n+1])$ to be the matrix
\begin{align*}
    f_1([n,n+1]) = \exp\left( \sum_j A_i (x^{(j)}_n - x^{(j)}_{n-1}) \right).
\end{align*}

We observe that the case of target categories $\D$ with more than one object
seems not to be considered in the data science literature.
We present one example to show that it might be worth considering
nonetheless.

\newcommand\MAT{\underline{\mathbf{Mat}}}
\begin{example}
    \label{ex:multi_object_target}

    Consider $\D = \MAT$ the category of matrices
    over some semiring\footnote{
        A semiring is a ``ring without the necessity of subtraction''.
        Every ring is a semiring.
        The tropical semiring $(\R \cup \{ -\infty \}, \min, +)$
        is a (commutative) semiring that is not a ring.
     } $S$ (for example $S = \R$).
    \begin{align*}
        \obj(\MAT)       &:= \N, \\
        \hom_{\MAT}(m,n) &:= S^{n\times m},
    \end{align*}
    with $\id_n$ the identity matrix of size $n$
    and composition given by matrix multiplication.
    (We note that, in the case $S=\R$, the delooping of the group $\GL_n(\R)$ embeds into $\MAT$
     as a subset of the morphisms at the object $n$.)

     Let $x_i \in \R, i \in \Z$, some time-series.
     Assign to it some natural numbers $n_i \in \N$, $i \in \Z$.
     For example, we could take a local complexity measure,
     such as permutation entropy,
     of the time-series and map it to the natural numbers.
     Idea: ``the more complex the time-series is locally around $i$, the
     larger the $n_i$.''

     Moreover, assume a family of embedding maps
     \begin{align*}
        \phi_{n \times m}: \R \to S^{n\times m},
     \end{align*}
     is given, for example as neural networks.
    We then define a map of quivers
    \begin{align*}
        f: Q \to \FCQ( \MAT ) \\
        f_0(i) = n_i, \quad f_1([i,i+1]) = \phi_{n_{i+1}\times n_i}(x_n).
    \end{align*}
    Using the universal property of the free category,
    we obtain a unique functor $F: \INT \to \MAT$,
    which aggregates the $\phi_{n_{i+1}\times n_i}(x_n)$
    over larger intervals, and gives some elements
    \begin{align*}
        F([i,j]) \in S^{n_j \times n_i}.
    \end{align*}
    
    A remark for experts:
    this is a special case of a \emph{quiver representation} (\cite{schiffler2014quiver})
    of $\EI$.
\end{example}


\subsection{Parallel scan}
\label{ss:parallel_scan_1}
\begin{quote}
    \itshape
    {
    The parallel scan algorithm is a well-known algorithm
    for computing the prefix sum of a sequence of numbers.
    It can be generalized to the setting of \Cref{ss:aggregation_as_functor}.
    }
\end{quote}

Let $F: \INT \to \D$ be a functor.
Assuming constant cost (in time and memory) of composing morphisms
in $\D$,
the calculation of the value\footnote{The
choice of ``source object'' $0$ is just for notational convenience.}
\begin{align*}
    F( [0,n] ),
\end{align*}
for $0 \le n$ is possible
at time-cost $\O( n )$
and memory-cost $\O(1)$.
In fact one can get the entire sequence
\begin{align}
    \label{eq:Fmk}
    ( F([0,k]) )_{k \in [0,n]},
\end{align}
at (the same) time-cost $\O( n )$
and memory-cost $\O( n )$.
Indeed, $F([0,1]) = f([0,1])$ is given,
and for $k = [2,n]$
\begin{align*}
    F([0,k]) = F([0,k-1]) \COMP F([k-1,k]), 
\end{align*}
with the claimed time and memory cost.

Blelloch \cite{blelloch1990prefix} realized,
in the special case of an associative product
on a set (i.e. the special case of $\D$ having
just one object), that this procedure is 
parallelizable. Given $p$ parallel machines,
\eqref{eq:Fmk} is then calcuable 
at time-cost $\O( n/p + \log p )$
(while the memory cost stays at $\O(n)$).
This ``associative scan'' is part of most deep learning
frameworks as a differentiable operation.
In its simplest form it is the ``cumsum'' operation
for the associative operation of summing real numbers.
In \verb|JAX|, any associative operation can be used with
\verb|jax.lax.associative_scan|.
Modern, sub-quadratic transformer alternatives
based on state-space models, benefit tremendously
from parallel scan, see e.g.  Mamba \cite{gu2023mamba} and S5
\cite{smithsimplified}.

We sketch the algorithm in the case where $p = n/2$
\footnote{
   For the general case,
    see \cite[p.43]{blelloch1990prefix}
    The argument there is given for semigroups, but it also extrapolates to general categories.
    }
and where $n = 2^\ell$ for some integer $\ell$.

In the first part, the ``up-sweep'' %
all aggregations over the dyadic intervals
$[k 2^i, (k+1) 2^{i}), i = 0, \dots, \ell-1$,
$k=0, \dots, 2^{\ell-i}-1$, are calculated.
Each level of the up-sweep
combines neighboring intervals
of the previous level.
The aggregations on a fixed level
are performed in parallel.
See \Cref{fig:upsweep} for an example,
where for illustrative purposes
the morphisms are taking values
in the (delooping) of the monoid $(\N_{\ge0},+)$.
For orientation, the nodes habe been labeled.
But remember that in the delooping of the monoid
there is only one object, the ``unit'' object $*$.

\tikzset{
    element/.style={circle, draw, minimum size=0.6cm, inner sep=1pt},
    operation/.style={draw=none, font=\small},
    downarrow/.style={->, >=Stealth, thick, black},
    downarrowpink/.style={->, >=Stealth, thick, magenta},
    uparrow/.style={->, >=Stealth, thick, black},
    stepnum/.style={draw=none, text=green!60!black, font=\small}
}

\begin{figure}
    \centering
    \begin{subfigure}[b]{0.49\textwidth}
        \centering
        \begin{tikzpicture}[
            scale=0.6, 
            transform shape,
            node distance=1.5cm]
            \draw[thick] (-0.5,0) -- (-0.5,9);
            \node[rotate=90] at (-0.8,4.5) {UP};
            \draw[->, thick] (-0.5,9) -- (-0.5,9.5);
            
            \node[element] (a1) at (0,0) {a};
            \node[element] (b1) at (1,0) {b};
            \node[element] (c1) at (2,0) {c};
            \node[element] (d1) at (3,0) {d};
            \node[element] (e1) at (4,0) {e};
            \node[element] (f1) at (5,0) {f};
            \node[element] (g1) at (6,0) {g};
            \node[element] (h1) at (7,0) {h};
            \node[element] (i1) at (8,0) {i};
            
            \draw[uparrow] (a1) to[bend left=30] node[black, midway, above] {3} (b1);
            \draw[uparrow] (b1) to[bend left=30] node[black, midway, above] {1} (c1);
            \draw[uparrow] (c1) to[bend left=30] node[black, midway, above] {7} (d1);
            \draw[uparrow] (d1) to[bend left=30] node[black, midway, above] {0} (e1);
            \draw[uparrow] (e1) to[bend left=30] node[black, midway, above] {4} (f1);
            \draw[uparrow] (f1) to[bend left=30] node[black, midway, above] {1} (g1);
            \draw[uparrow] (g1) to[bend left=30] node[black, midway, above] {6} (h1);
            \draw[uparrow] (h1) to[bend left=30] node[black, midway, above] {3} (i1);
            
            \node[element] (a2) at (0,2.5) {a};
            \node[element] (c2) at (2,2.5) {c};
            \node[element] (e2) at (4,2.5) {e};
            \node[element] (g2) at (6,2.5) {g};
            \node[element] (i2) at (8,2.5) {i};
            
            \draw[decorate, decoration={brace, amplitude=5pt}] (0.1,1.0) -- (1.9,1.0);
            \draw[uparrow, thick] (1,1.15) -- (1,2.35);
            
            \draw[decorate, decoration={brace, amplitude=5pt}] (2.1,1.0) -- (3.9,1.0);
            \draw[uparrow, thick] (3,1.15) -- (3,2.35);
            
            \draw[decorate, decoration={brace, amplitude=5pt}] (4.1,1.0) -- (5.9,1.0);
            \draw[uparrow, thick] (5,1.15) -- (5,2.35);
            
            \draw[decorate, decoration={brace, amplitude=5pt}] (6.1,1.0) -- (7.9,1.0);
            \draw[uparrow, thick] (7,1.15) -- (7,2.35);
            
            \draw[uparrow] (a2) to[bend left=30] node[stepnum, midway, above] {4} (c2);
            \draw[uparrow] (c2) to[bend left=30] node[stepnum, midway, above] {7} (e2);
            \draw[uparrow] (e2) to[bend left=30] node[stepnum, midway, above] {5} (g2);
            \draw[uparrow] (g2) to[bend left=30] node[stepnum, midway, above] {9} (i2);
            
            \node[element] (a3) at (0,5) {a};
            \node[element] (e3) at (4,5) {e};
            \node[element] (i3) at (8,5) {i};
            
            \draw[decorate, decoration={brace, amplitude=5pt}] (0.1,3.7) -- (3.9,3.7);
            \draw[uparrow, thick] (2,3.85) -- (2,4.85);
            
            \draw[decorate, decoration={brace, amplitude=5pt}] (4.1,3.7) -- (7.9,3.7);
            \draw[uparrow, thick] (6,3.85) -- (6,4.85);
            
            \draw[uparrow] (a3) to[bend left=30] node[stepnum, midway, above] {11} (e3);
            \draw[uparrow] (e3) to[bend left=30] node[stepnum, midway, above] {14} (i3);
            
            \node[element] (a4) at (0,7.5) {a};
            \node[element] (i4) at (8,7.5) {i};
            
            \draw[decorate, decoration={brace, amplitude=5pt}] (0.1,6.3) -- (7.9,6.3);
            \draw[uparrow, thick] (4,6.45) -- (4,7.35);
            
            \draw[uparrow] (a4) to[bend left=40] node[stepnum, midway, above] {25} (i4);
        \end{tikzpicture}
        \caption{Up-sweep phase of parallel scan; green labels and braces indicate aggregation}
        \label{fig:upsweep}
    \end{subfigure}
    \hfill
    \begin{subfigure}[b]{0.49\textwidth}
        \centering
        \begin{tikzpicture}[
            scale=0.6, 
            transform shape,
            node distance=1.5cm]
            \draw[thick] (-0.5,0) -- (-0.5,9);
            \node[rotate=90] at (-0.8,4.5) {DOWN};
            \draw[->, thick] (-0.5,0) -- (-0.5,-0.5);
            
            \node[element] (a4) at (0,7.5) {a};
            \node[element] (i4) at (8,7.5) {i};
            
            \draw[->, thick] (a4) to[out=45, in=135, looseness=5] node[stepnum, above] {0} (a4);
            
            \draw[->, thick] (a4) to[bend left=20] node[stepnum, midway, above] {25} (i4);
            
            \node[element] (a3) at (0,5) {a};
            \node[element] (e3) at (4,5) {e};
            \node[element] (i3) at (8,5) {i};
            
            \draw[->, thick] (a3) to[out=45, in=135, looseness=5] node[left] {0} (a3);
            
            \draw[->, thick] (a3) to[bend left=20] node[stepnum, midway, above] {11} (e3);
            
            \draw[->, thick] (a3) to[bend left=30] node[black, midway, above] {25} (i3);
            
            \node[element] (a2) at (0,2.5) {a};
            \node[element] (c2) at (2,2.5) {c};
            \node[element] (e2) at (4,2.5) {e};
            \node[element] (g2) at (6,2.5) {g};
            \node[element] (i2) at (8,2.5) {i};
            
            \draw[->, thick] (a2) to[out=45, in=135, looseness=5] node[left] {0} (a2);
            
            \draw[->, thick] (a2) to[bend left=20] node[stepnum, midway, above] {4} (c2);
            
            \draw[->, thick] (a2) to[bend left=25] node[black, midway, above] {11} (e2);
            
            \draw[->, thick] (a2) to[bend left=30] node[stepnum, midway, above] {16} (g2);
            
            \draw[->, thick] (a2) to[bend left=35] node[black, midway, above] {25} (i2);
            
            \node[element] (a1) at (0,0) {a};
            \node[element] (b1) at (1,0) {b};
            \node[element] (c1) at (2,0) {c};
            \node[element] (d1) at (3,0) {d};
            \node[element] (e1) at (4,0) {e};
            \node[element] (f1) at (5,0) {f};
            \node[element] (g1) at (6,0) {g};
            \node[element] (h1) at (7,0) {h};
            \node[element] (i1) at (8,0) {i};
            
            \draw[->, thick] (a1) to[out=45, in=135, looseness=5] node[left] {0} (a1);
            
            \draw[->, thick] (a1) to[bend left=25] node[black, pos=0.8, above] {4} (c1);
            \draw[->, thick] (a1) to[bend left=30] node[stepnum, pos=0.8, above] {11} (d1);
            \draw[->, thick] (a1) to[bend left=35] node[black, pos=0.8, above] {11} (e1);
            \draw[->, thick] (a1) to[bend left=40] node[stepnum, pos=0.8, above] {15} (f1);
            \draw[->, thick] (a1) to[bend left=45] node[black, pos=0.8, above] {16} (g1);
            \draw[->, thick] (a1) to[bend left=50] node[stepnum, pos=0.8, above] {22} (h1);
            \draw[->, thick] (a1) to[bend left=55] node[black, pos=0.8, above] {25} (i1);
            \draw[->, thick] (a1) to[bend left=20] node[stepnum, pos=0.8, above, fill=white, fill opacity=0.7, ] {3} (b1);
        \end{tikzpicture}
        \caption{Down-sweep phase of parallel scan; green labels indicate computation}
        \label{fig:downsweep}
    \end{subfigure}
    \caption{Parallel scan algorithm phases}
    \label{fig:parallel-scan}
\end{figure}
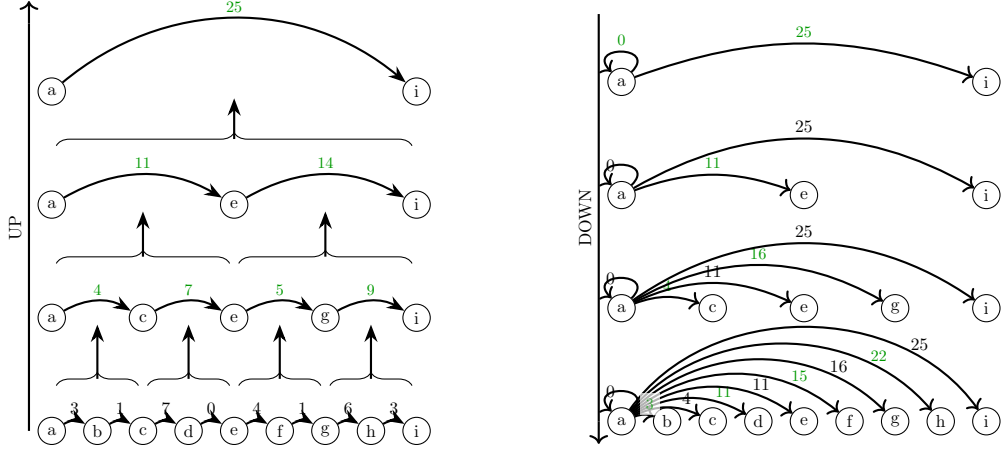

In the second step, the ``down-sweep'',
the results of the ``up-sweep'' are used
to calculate the values
on all intervals $[0,k]$, $k=0, \dots, n$.
Again, on a fixed level, the calculations
are performed in parallel.
See \Cref{fig:downsweep} for an example,
continuing from \Cref{fig:upsweep}.

%
%
%


\section{Two parameters: spatial data}
\label{sec:two-param}

Let spatial data $x_\k$, $\k \in \Z^2$, be given.
For example $x_\k \in \R$.

We want to aggregate this data over rectangles
$[0,m] \times [0,n]$.
In the example of real-valued data, simple examples are
the sum and the max over the rectangle,
\begin{align*}
    \sum_{k=0}^m \sum_{\ell=0}^n x_{k,\ell}, \qquad
    \max_{k=0,\dots,m} \max_{\ell=0,\dots,n} x_{k,\ell}.
\end{align*}

As in the one-parameter case,
these are examples of a general procedure
of categorial aggregation, now using double functors.


\subsection{Double categories, double functors and free double categories}
\label{ss:double_categories_double_functors}

\begin{quote}
    \itshape
    {
    A \emph{double category} is set of 0-cells (objects), 1-cells (vertical and horizontal morphisms), and 2-cells (faces) together with composition maps that satisfy the axioms of a category in both the vertical and horizontal direction,
    and are compatible with each other in a certain way.
    A \emph{double functor} is a structure preserving map between double categories.
    A certain \emph{double category of rectangles} is the free double category over elementary rectangles.
    }
\end{quote}

We will use \emph{heuristic working definitions} of several higher categorical structures.
Their precise definition can be found in \Cref{sec:appendix}.

A \DEF{double category} $\DCC$ consists of
\begin{itemize}
    \item a set of \DEF{$0$-cells} $\DCC_0$,
    \item a set of \DEF{horizontal $1$-cells} $\DCC^h_1$,
    \item a set of \DEF{vertical $1$-cells} $\DCC^v_1$,
    \item a set of \DEF{$2$-cells} $\DCC_2$,
\end{itemize}
with various source and target maps
($\partial^\pm_v, \dots$)
as well as identity maps ($\i_v, \dots$)
and composition maps ($\COMP_v, \dots$).
A $1$-cell $f$ can be thought of as a vertical or horizontal line segment with corners in $\DCC_0$.

\begin{figure}[H]
    \centering
\begin{tikzpicture}[
    point/.style={circle, fill=black, inner sep=2pt},
    >=Stealth
]
    \node[point] (A) at (0,0) {};
    \node[point] (B) at (4,0) {};
    
    \draw[->, very thick, blue] (A) to node[midway, above] {$f$} (B);
    
    \node[gray] at (0,-0.5) {$\partial^-_h(f)$};
    \node[gray] at (4,-0.5) {$\partial^+_h(f)$};
    
\end{tikzpicture}
\end{figure}

A $2$-cell $\alpha$ can be thought of as a rectangle with corners in $\DCC_0$,
and two vertical $1$-cells
($\partial^-_H(\alpha), \partial^+_H(\alpha)$) 
and two horizontal $1$-cells as edges
($\partial^-_V(\alpha), \partial^+_V(\alpha)$).

\begin{figure}[H]
    \centering
    \begin{tikzpicture}[
        point/.style={circle, fill=black, inner sep=2pt},
        hcell/.style={blue, very thick},
        vcell/.style={red, very thick},
        >=Stealth
    ]
        \fill[purple!15] (0,0) rectangle (4,3);
        
        \node[point] (A) at (0,0) {};
        \node[point] (B) at (4,0) {};
        \node[point] (C) at (0,3) {};
        \node[point] (D) at (4,3) {};
        
        \draw[->,hcell] (A) to node[midway, below] {$\partial^-_V(\alpha)$} (B);
        \draw[->,hcell] (C) to node[midway, above] {$\partial^+_V(\alpha)$} (D);
        
        \draw[->,vcell] (A) to node[midway, left] {$\partial^-_H(\alpha)$} (C);
        \draw[->,vcell] (B) to node[midway, right] {$\partial^+_H(\alpha)$} (D);
        
        \node[purple] at (2,1.5) {$\alpha$};
    \end{tikzpicture}
    \caption{A $2$-cell $\alpha$ with its boundary}
\end{figure}
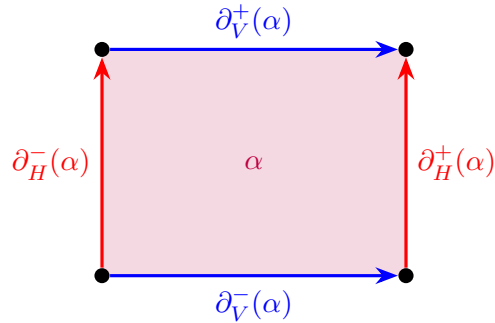

Two $2$-cells $\alpha,\beta$ can be composed horizontally if the 
vertical edge of $\alpha$ ``to the right'' (the horizontal target) is the same as the vertical edge of $\beta$ ``to the left''
(the horizontal source).
We denote this composition by $\alpha \COMP_H \beta$.
Analogously, the vertical composition of $2$-cells is denoted $\alpha \COMP_V \beta$.

Two vertical $1$-cells can be composed if the ``top corner'' (the target) of the first is the same as the ``bottom corner'' (the source) of the second.
We denote this composition by $\alpha \COMP_v \beta$.
Analogously, the horizontal composition of $1$-cells is denoted $\alpha \COMP_h \beta$.

\begin{figure}[H]
    \centering
    \begin{tikzpicture}[
        scale=0.9, 
        transform shape,
        point/.style={circle, fill=black, inner sep=2pt},
        hcell/.style={blue, very thick},
        vcell/.style={red, very thick},
        >=Stealth
    ]
        
        \fill[purple!15] (0,3) rectangle (6,6);
        \fill[cyan!15] (0,0) rectangle (6,3);
        \node[point] (A) at (0,6) {};
        \node[point] (B) at (6,6) {};
        \node[point] (C) at (0,3) {};
        \node[point] (D) at (6,3) {};
        
        \draw[->,hcell] (C) to node[midway, below] {$g$} (D);
        \draw[->,hcell] (A) to node[midway, above] {$h$} (B);
        
        \draw[->,vcell] (A) to node[midway, left] {$w$} (C);
        \draw[->,vcell] (B) to node[midway, right] {$x$} (D);
        
        \node[purple] at (3,4.5) {$\beta$};
        
        \node[point] (E) at (0,0) {};
        \node[point] (F) at (6,0) {};
        
        \draw[->,hcell] (E) to node[midway, below] {$f$} (F);
        \draw[->,hcell] (C) to node[midway, above] {$g$} (D);
        
        \draw[->,vcell] (E) to node[midway, left] {$u$} (C);
        \draw[->,vcell] (F) to node[midway, right] {$v$} (D);
        
        \node[cyan] at (3,1.5) {$\alpha$};
        
        \draw[|->, thick] (7,3) -- (7.8,3);
        
        \fill[orange!15] (9,0) rectangle (15,6);
        \node[point] (G) at (9,6) {};
        \node[point] (H) at (15,6) {};
        \node[point] (I) at (9,0) {};
        \node[point] (J) at (15,0) {};
        
        \draw[->,hcell] (I) to node[midway, below] {$f$} (J);
        \draw[->,hcell] (G) to node[midway, above] {$h$} (H);
        
        \draw[<-,vcell] (G) to node[midway, left] {$u \COMP_v w$} (I);
        \draw[<-,vcell] (H) to node[midway, right] {$v \COMP_v x$} (J);
        
        \node[orange] at (12,3) {$\alpha \COMP_V \beta$};
        
    \end{tikzpicture}
    \caption{Two 2-cells $\alpha$ and $\beta$ can be composed vertically when the bottom horizontal 1-cell of $\alpha$ 
           is the same as the top horizontal 1-cell of $\beta$ (here, $g$).}
    \label{fig:vertical_composition}
\end{figure}
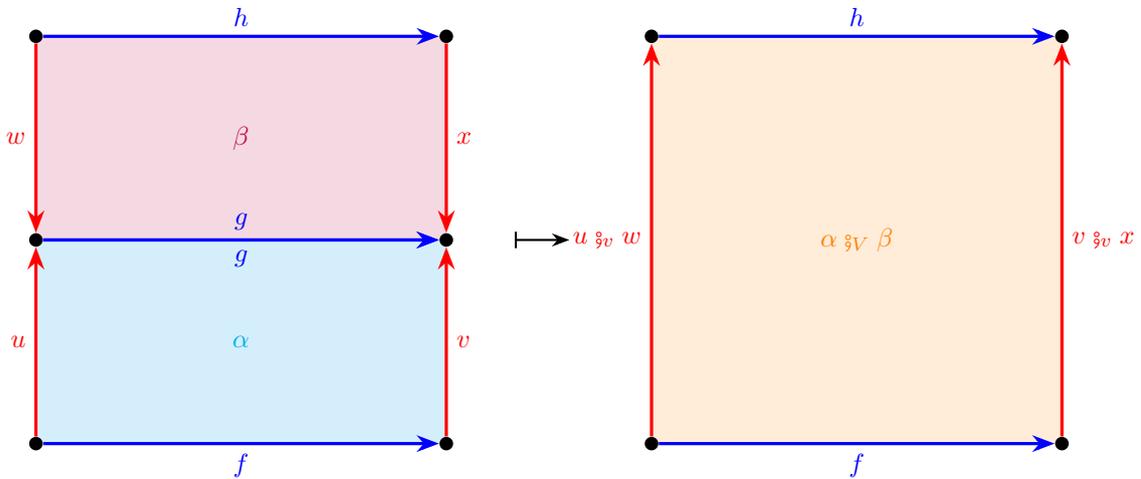

\begin{figure}
    \centering
    \begin{tikzpicture}[
        scale=0.9, 
        transform shape,
        point/.style={circle, fill=black, inner sep=2pt},
        hcell/.style={blue, very thick},
        vcell/.style={red, very thick},
        >=Stealth
    ]
        
        \fill[purple!15] (0,0) rectangle (3,3);
        \fill[cyan!15] (3,0) rectangle (6,3);
        \node[point] (A) at (0,0) {};
        \node[point] (B) at (3,0) {};
        \node[point] (C) at (0,3) {};
        \node[point] (D) at (3,3) {};
        
        \draw[->,hcell] (A) to node[midway, below] {$f$} (B);
        \draw[->,hcell] (C) to node[midway, above] {$g$} (D);
        
        \draw[->,vcell] (A) to node[midway, left] {$u$} (C);
        \draw[->,vcell] (B) to node[midway, right] {$v$} (D);
        
        \node[purple] at (1.5,1.5) {$\alpha$};
        
        \node[point] (E) at (6,0) {};
        \node[point] (F) at (6,3) {};
        
        \draw[->,hcell] (B) to node[midway, below] {$h$} (E);
        \draw[->,hcell] (D) to node[midway, above] {$i$} (F);
        
        \draw[->,vcell] (E) to node[midway, right] {$w$} (F);
        
        \node[cyan] at (4.5,1.5) {$\beta$};
        
        \node[above] at (3,1.5) {$v$};
        
        \draw[|->, thick] (7,1.5) -- (8,1.5);
        
        \fill[orange!15] (9,0) rectangle (15,3);
        \node[point] (G) at (9,0) {};
        \node[point] (H) at (15,0) {};
        \node[point] (I) at (9,3) {};
        \node[point] (J) at (15,3) {};
        
        \draw[->,hcell] (G) to node[midway, below] {$f \COMP_h h$} (H);
        \draw[->,hcell] (I) to node[midway, above] {$g \COMP_h i$} (J);
        
        \draw[->,vcell] (G) to node[midway, left] {$u$} (I);
        \draw[->,vcell] (H) to node[midway, right] {$w$} (J);
        
        \node[orange] at (12,1.5) {$\alpha \COMP_H \beta$};
        
    \end{tikzpicture}
    \caption{Two 2-cells $\alpha$ and $\beta$ can be composed horizontally when the right vertical 1-cell of $\alpha$ 
            is the same as the left vertical 1-cell of $\beta$ (here, $v$).};
    \label{fig:horizontal_composition}
\end{figure}
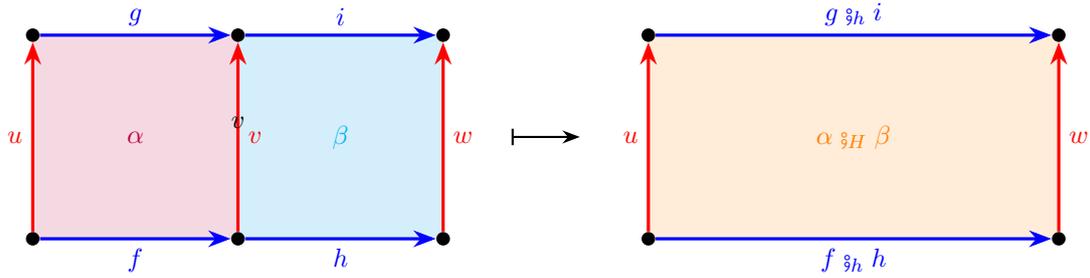

$0$-cells have horizontal and vertical identity $1$-cells attached to them,
vertical $1$-cells have horizontal $2$-cells attached to them,
and horizontal $1$-cells have vertical $2$-cells attached to them.
They all act as identities under the respective compositions.

All compositions are associative. This turns the horizontal $1$-cells (together with the $0$-cells)
into a (classical) category (analogously for the vertical $1$-cells).
Likewise, the $2$-cells together with vertical $1$-cells form a category (analogously for the horizontal $1$-cells
and vertical $1$-cells).

A truely ``higher categorial'' condition is given by the interchange law:

\definecolor{lightgreen}{RGB}{210,255,170}
\definecolor{lightpink}{RGB}{255,180,255}

\begin{center}
\begin{tikzpicture}[scale=1.2]
    \begin{scope}[shift={(-2.5,0)}]
        \fill[lightgreen] (0,0) rectangle (1,1);
        \fill[lightgreen] (0,1) rectangle (1,2);
        \fill[lightpink] (1,0) rectangle (2,1);
        \fill[lightpink] (1,1) rectangle (2,2);
        
        \draw[thick] (0,0) rectangle (2,2);
        \draw[thick] (1,0) -- (1,2);
        \draw[dotted, thick] (0,1) -- (2,1);
        
        \draw[dotted, thick] (0,1) -- (2,1);
        \draw[dotted, thick] (1,0) -- (1,2);
        
        \node at (0.5,0.5) {$\alpha$};
        \node at (0.5,1.5) {$\gamma$};
        \node at (1.5,0.5) {$\beta$};
        \node at (1.5,1.5) {$\delta$};
        \node at (1,-1) {$\greenbox{$(\alpha \COMP_V \gamma)$} \COMP_H \pinkbox{$(\beta \COMP_V \delta)$}$};
    \end{scope}
    
    \begin{scope}[shift={(2.5,0)}]
        \fill[lightgreen] (0,0) rectangle (1,1);
        \fill[lightpink] (0,1) rectangle (1,2);
        \fill[lightgreen] (1,0) rectangle (2,1);
        \fill[lightpink] (1,1) rectangle (2,2);
        
        \draw[thick] (0,0) rectangle (2,2);
        \draw[dotted, thick] (1,0) -- (1,2);
        \draw[thick] (0,1) -- (2,1);
        
        \draw[dotted, thick] (0,1) -- (2,1);
        \draw[dotted, thick] (1,0) -- (1,2);
        
        \node at (0.5,0.5) {$\alpha$};
        \node at (0.5,1.5) {$\gamma$};
        \node at (1.5,0.5) {$\beta$};
        \node at (1.5,1.5) {$\delta$};
        
        \node at (1,-1) {$\greenbox{$(\alpha \COMP_H \beta)$} \COMP_V \pinkbox{$(\gamma \COMP_H \delta)$}$};
    \end{scope}
    
    \node at (1,-1) {$=$};
    \end{tikzpicture}
\end{center}

It states that the composition of four composable $2$-cells in the shown constellation
is independent of the order of composition.

\renewcommand\P{{\mathbb{P}}}
\begin{example}[{Rectangles; \cite[Section A.7]{SoncZucc15}, \cite[Example 2.10]{fiore2008model}.}]
    Define the double category $\RECT$ of rectangles with corners in $\Z^2$ as follows.
    \begin{align*}
        \RECT_0    &:= \Z^2 \\
        \RECT^h_1  &:= \{ [s_1,t_1] \times \{x_2\} \mid s_1,t_1,x_2 \in \Z, s_1 \le t_1 \} \\
        \RECT^v_1  &:= \{ \{x_1\} \times [s_1,t_1] \mid x_1,s_2,t_2 \in \Z, s_2 \le t_2 \} \\
        \RECT_2    &:= \{ [s_1,t_1] \times [s_2,t_2] \mid s_1,t_1,s_2,t_2 \in \Z, s_1 \le t_1, s_2 \le t_2 \},
    \end{align*}
    with
    \begin{align*}
        \partial^-_h( [s_1,t_1] \times \{x_2\} ) &:= (s_1,x_2) \\
        \partial^+_h( [s_1,t_1] \times \{x_2\} ) &:= (t_1,x_2) \\
        \i_{h}(x_1,x_2) &:= [x_1,x_1] \times \{x_2\},
    \end{align*}
    and analogously for the vertical morphisms.

    Further
    \begin{align*}
        \partial^-_V( [s_1,t_1] \times [s_2,t_2]  ) &:= [s_1,t_1] \times \{s_2\}, \\
        \partial^+_V( [s_1,t_1] \times [s_2,t_2]  ) &:= [s_1,t_1] \times \{t_2\}, \\
        \partial^-_H( [s_1,t_1] \times [s_2,t_2]  ) &:= \{s_1\} \times [s_2,t_2], \\
        \partial^+_H( [s_1,t_1] \times [s_2,t_2]  ) &:= \{t_1\} \times [s_2,t_2], \\
        \i_{V}( [s_1,t_1] \times \{x_2\} ) &:= [s_1,t_1] \times [x_2,x_2], \\
        \i_{H}( [s_1,t_1] \times \{x_2\} ) &:= [s_1,t_1] \times [x_2,x_2].
    \end{align*}

    The compositions are all given by set union,
    so that \uline{in this double category the figures above can be taken literally}.

\end{example}

This double category of rectangles will be the \emph{domain} of our double functors built from spatial data.

Regarding the \emph{codomain} of our double functors,
if we try to find an analog of the delooping of a monoid
\Cref{ex:categories}, \ref{ex:categories:delooping},
we find that, at first sight, only commutative monoids
can be used.
\begin{example}
    \label{ex:abelian_group}

    Let $A$ be an abelian group (or abelian monoid).
    Define the double category $\DCC$ as follows: 
    \begin{itemize}
        \item $\DCC_0 = \{*\}$,
        \item $\DCC^h_1 = \DCC^v_1 = \{*\}$,
        \item $\DCC_2 = A$,
    \end{itemize}
    with obvious structure maps.
    Only the interchange law needs to be checked,
    which follows from (and necessitates) the commutativity of $A$.

    Note that the 
    two categories of $1$-cells are the same. 
    Such a double category is called \DEF{edge-symmetric} (\cite{brown1999double}).

\end{example}
If the group in \Cref{ex:abelian_group} is not abelian,
the interchange law will not hold
(this can be seen by the \emph{Eckman-Hilton argument}, see \cite{eckmann1962group}).
In order to get interesting (non-abelian)
codomains for our double functors,
we need the concept of a \emph{crossed module of groups}.

\newcommand\conj{\mathrm{conj}}
\begin{definition}[Whitehead '49 \cite{whitehead1949combinatorial}]
    A \textbf{crossed module of groups} $(\feedbackGG: H \to G, \actionGG)$
    is a diagram
    \begin{align*}
        H \stackrel{\feedbackGG}\to G \stackrel{\actionGG}\to \Aut(H),
    \end{align*}
    where $H,G$ are groups and $\feedbackGG$ and $\actionGG$ are group morphisms, satisfying the identities
    \begin{align}
        \feedbackGG \circ \actionGG_g    &= \conj_g \circ \feedbackGG && g \in G, \label{eq:equivariant} \tag{EQUI}\\
        \actionGG_{\feedbackGG(h)}       &= \conj_h     && h \in H. \label{eq:peiffer} \tag{PEIF}
    \end{align}
    Here, $\conj$ denotes the inner action of a group on itself,
    i.e. $\conj_g(g') = g g' g^{-1}$.
  
    $\feedbackGG$ is called the \textbf{feedback} and $\actionGG$ the \textbf{action}.
\end{definition}

\begin{example}~
    \label{ex:crossed_modules}
    \begin{enumerate}

        \item
        Any normal subgroup $H \trianglelefteq G$ of a group $G$ gives rise to a crossed module
        $(\feedbackGG: H \to G, \actionGG)$,
        where $\feedbackGG$ is the inclusion and $\actionGG$ is the conjugation action of $G$ on $H$.
        This example is ``trivial'' in the sense that $\feedbackGG$ is injective.

        \item
        Any group $G$ gives a crossed module
        $(\feedbackGG: G \to \Aut(G), \actionGG)$,
        where $\feedbackGG$ is taking a group element to the corresponding inner automorphism,
        i.e. $\feedbackGG(g) = \conj_g$,
        and $\actionGG$ is given by evaluation of an automorphism on the group elements.

        \item
        \label{ex:crossed_modules:abelian}
        Let $A$ be an abelian group;
        then $\feedbackGG: A \to 1$ is a crossed module,
        where $\feedbackGG$ is the trivial map
        and $\actionGG$ is given by the trivial action.
        (This is a sub-crossed module of the previous one.)
        
        \item
        \label{ex:crossed_modules:general_linear}
        The general linear crossed module.%
        \footnote{
                \cite[Section 2.2]{martins2011lie},
                \cite{forrester2003representations},
                \cite[Section 4.2]{lee2023random}
            }
        Fix a field $\K$, integers $n,p,q \ge 0$ and consider 
        \begin{align*}
            G
            &:=
            \{ 
                (
            \begin{bmatrix}
                P & 0_{n\times p} \\
                R & S
            \end{bmatrix},
            \begin{bmatrix}
                P & B \\
                0_{q\times n} & D
            \end{bmatrix} )
            \mid
            P \in \GL_n(\K), R \in \K^{p \times n}, S \in \GL_p(\K), \\
            &\qquad\qquad
            \qquad\qquad
            \qquad\qquad
            \qquad
            \quad
            B \in \K^{n \times q}, D \in \GL_q(\K) \}.
        \end{align*}
        This is a group via entrywise matrix multiplication.
        Let
        \begin{align*}
            H
            :=
            \GL^{n,p,q}_{-1}
            := 
            \{ 
                \begin{bmatrix}
                    P - \id_n & B \\
                    R & N
                \end{bmatrix} \mid
                P \in \GL_n(\K), R \in \K^{p \times n}, N \in \K^{p \times q}, B \in \K^{n \times q} \}.
        \end{align*}
        It becomes a group under the operation
        \begin{align*}
            \begin{bmatrix}
                P-\id_n & B \\
                R & N
            \end{bmatrix}
            \bullet_h
            \begin{bmatrix}
                P'-\id_n & B' \\
                R' & N'
            \end{bmatrix}
            =
            \begin{bmatrix}
                P'P - \id_n & P' B + B' \\
                R' P + R & R' B + N + N'
            \end{bmatrix}.
        \end{align*}
        The unit is given by the zero matrix, and 
        the inverse is given by
        \begin{align*}
            \begin{bmatrix}
                P - \id_n & B \\
                R & N
            \end{bmatrix}^{\bullet_h -1}
            =
            \begin{bmatrix}
                P^{-1} - \id_n & -P^{-1} B \\
                -R P^{-1} & -N + R P^{-1} B
            \end{bmatrix}.
        \end{align*}

        The feedback is defined as follows
        \begin{align*}
            \feedbackGG( 
            \begin{bmatrix}
                P - \id_n & B \\
                R & N
            \end{bmatrix} )
            :=
            (
                \begin{bmatrix}
                P & 0_{n\times p} \\
                R & \id_p
                \end{bmatrix},
                \begin{bmatrix}
                    P & B \\
                    0_{q\times n} & \id_q
                \end{bmatrix}
            )
        \end{align*}
        The action of $(f_U, f_V) \in G$ on $h \in H$ is given by
        \begin{align*}
            \triangleright_{ (f_U, f_V) }( h )
            =
            f_V^{-1} \cdot h \cdot f_U.
        \end{align*}

        \item
        A certain free crossed module of Lie algebras
        is constructed in \cite{kapranov2015membranes},
        which formally corresponds to a (Lie) group.
        See also \cite{chevyrev2024multiplicative,lee2024_2Dsigs}.

    \end{enumerate}
    
\end{example}

The following double category will usually form the codomain
of our functors built from spatial data.
It can be considered the delooping of a crossed module
to a double category%
\footnote{ In fact, a double \emph{groupoid}, but this is not important for us.  }.
\begin{example}[{Crossed module to double category; \cite[Section A.8]{SoncZucc15}}]
	\label{ex:cm_to_double}

	Let $\mathfrak C = (\feedbackGG: H \to G, \actionGG)$ be a crossed module of groups.
	Define 
    the following edge-symmetric double category $\BB \mathfrak C$.
    The category of $1$-cells is the delooping of $G$,
    in particular $D_0 := \{*\}$, $D_1 := G$.
    Further,
	\begin{align*}
		D_2 := \{
            \SQUAREPATTERN{x_s}{x_e}{x_n}{x_w}{X} \in G^4 \times H \mid \feedbackGG(X) x_w x_n = x_s x_e \},
	\end{align*}
	with boundaries
	\begin{align*}
		\partial_V^-(\SQUAREPATTERN{x_s}{x_e}{x_n}{x_w}{X}) &:= x_s, & \quad
        \partial_V^+(\SQUAREPATTERN{x_s}{x_e}{x_n}{x_w}{X}) := x_n, \\
        \partial_H^-(\SQUAREPATTERN{x_s}{x_e}{x_n}{x_w}{X}) &:= x_w, & \quad
        \partial_H^+(\SQUAREPATTERN{x_s}{x_e}{x_n}{x_w}{X}) := x_e.
	\end{align*}

	Horizontal composition of $2$-cells is given by (if $x_e = y_w$)
	\begin{align*}
        \SQUAREPATTERN{x_s}{x_e}{x_n}{x_w}{X} \COMP_H \SQUAREPATTERN{y_s}{y_e}{y_n}{y_w}{Y}
        :=
        \SQUAREPATTERN{x_s y_s}{y_e}{x_n y_n}{x_w}{ \actionGG_{x_s}(Y) X }.
	\end{align*}
    Vertical composition is given by (if $x' = b$)
    \begin{align*}
        \SQUAREPATTERN{x_s}{x_e}{x_n}{x_w}{X} \COMP_V \SQUAREPATTERN{x'_s}{x'_e}{x'_n}{x'_w}{X}
        :=
        \SQUAREPATTERN{x_s}{x_e x'_e}{x'_n}{x_w x'_w}{ X \actionGG_{x_w}(X') }.
    \end{align*}
    We verify the double category axioms for this construction in \Cref{ex:full_cm_to_double}.
\end{example}

A \DEF{double functor} $F: \DCC \to \DCD$ between double categories
is given by four maps $F_i: C_i \to D_i$, $i=0,2$, $F_1^h: C_1^h \to D_1^h$, $F_1^v: C_1^v \to D_1^v$,
which commute with all the structure maps of the double category.

\begin{example}
    \label{ex:double_sum}

        Let $x_\k \in \R$, $\k \in \Z^2$, be spatial data.
        Let $\DCC$ 
        be the double category from 
        \Cref{ex:abelian_group} for the abelian group $(\R,+)$.
        Define
        \begin{align*}
            F_0(\k)                        &:= *, \\
            F^h_1([s_1,t_1]\times \{x_2\}) &:= *, \\
            F^v_1(\{x_1\}\times [s_2,t_2]) &:= *, \\
            F_2([s_1,t_1]\times[s_2,t_2])  &:= \sum_{\k \in [s_1,t_1]\times[s_2,t_2]} x_\k.
        \end{align*}

        Then $F: \RECT \to \DCC$ is a double functor.

    
\end{example}

As in the case of functors on $\INT$,
we prove the functoriality in the example
by establishing freeness of the double category $\RECT$.
This then enables us to aggregate in general double categories,
for example in the double category of crossed modules \Cref{ex:cm_to_double}.

A \DEF{double graph} $\DGG$ consists of
\begin{itemize}
    \item a set of \DEF{$0$-cells} $\DGG_0$,
    \item a set of \DEF{horizontal $1$-cells} $\DGG^h_1$,
    \item a set of \DEF{vertical $1$-cells} $\DGG^v_1$,
    \item a set of \DEF{$2$-cells} $\DGG_2$,
\end{itemize}
with various source and target maps:
\begin{align*}
    \partial^\pm_h: \DGG^h_1 \to \DGG_0, & \quad
    \partial^\pm_v: \DGG^v_1 \to \DGG_0, \\
    \partial^\pm_H: \DGG_2 \to \DGG^v_1, & \quad
    \partial^\pm_V: \DGG_2 \to \DGG^h_1.
\end{align*}
The only consistency condition is that ``the four corners of a $2$-cell are well-defined'':
\begin{align*}
    \partial^-_v( \partial^-_H(\alpha) ) &= \partial^-_h( \partial^-_V(\alpha) ), \quad
    \partial^+_v( \partial^-_H(\alpha) )  = \partial^-_h( \partial^+_V(\alpha) ) \\
    \partial^-_v( \partial^+_H(\alpha) ) &= \partial^+_h( \partial^-_V(\alpha) ), \quad
    \partial^+_v( \partial^+_H(\alpha) ) = \partial^+_h( \partial^+_V(\alpha) ).
\end{align*}
A \DEF{morphism of double graphs} $F: \DGG \to \DGH$ consists of maps
\begin{align*}
    F_0: \DGG_0 \to \DGH_0, & \quad
    F^h_1: \DGG^h_1 \to \DGH^h_1, \\
    F^v_1: \DGG^v_1 \to \DGH^v_1, & \quad
    F_2: \DGG_2 \to \DGH_2,
\end{align*}
which respect the boundary maps.
We note that any double category $\DCD$ can be considered as a double graph
by ``forgetting compositions and identities''. A double functor becomes a morphism of 
the underlying double graphs.

\bigskip

The following theorem is essentially contained in \cite{dawson2002free}.
Our proof is inspired by \cite{dawson2002free}.
\begin{theorem}
    \label{thm:free_double_category}
    $\RECT$ is the free double category    
    over the following double
    graph $R$
    \begin{align*}
        R_0   &:= \Z^2 \\
        R^h_1 &:= \{ [i,i+1]\times\{j\} \mid i,j \in \Z \}   \\
        R^v_1 &:= \{ \{i\} \times [j,j+1] \mid i,j \in \Z \} \\
        R_2   &:= \{ [i, i+1] \times [j, j+1] \mid i,j \in \Z \},
    \end{align*}
    with obvious boundaries.
    This means that there is a morphism of double graphs $\iota: R \to \RECT$
    such that for any double category $\mathbb D$ and
    any morphism of double graphs $F: R \to \mathbb D$,
    there exists a unique double functor $\hat F: \RECT \to \mathbb D$
    such that $\iota \COMP \hat F = F$.
\end{theorem}
\begin{proof}
    \newcommand\RR{{\mathcal{R}}}

    Let $\mathbb D$ be any double category and
    let $F: R \to \mathbb D$ be any morphism of double graphs.
    Let $\iota: R \to \RECT$ be the inclusion, which is clearly
    a morphism of double graphs.

    For horizontal $1$-cells we use the fact that the ``strip''
    \begin{align*}
        \{ [s_1,t_1] \times \{j\} \mid s_1,t_1 \in \Z \} \subset \RECT_1^h,
    \end{align*}
    is the free category on the elementary $1$-cells $[i,i+1] \times \{j\}$.
    Analogously for vertical $1$-cells.
    $F_1^h$ and $F_1^v$ then lift to unique functors 
    $\hat F_1^h$ and $\hat F_1^v$.

    Regarding $2$-cells, 
    on elementary squares we set
    \begin{align*}
        \hat F( [i,i+1] \times [j,j+1] ) := F( [i,i+1] \times [j,j+1] ).
    \end{align*}

    Write $\COMP^\RECT_H, \COMP^\RECT_V,$ for the composition of $2$-cells in $\RECT$,
    and $\COMP^\DCD_H, \COMP^\DCD_V$ for the composition of $2$-cells in $\mathbb D$.
    If $\RR$ is a non-elementary rectangle, it
    can be written as $\RR = \RR_A \COMP^\RECT_H \RR_B$
    (with $\RR_A,\RR_B$ non-degenerate)
    or $\RR = \RR_{A} \COMP^\RECT_V \RR_{A'}$ (with $\RR_A,\RR_{A'}$ non-degenerate).
    Note that both cases can happen and in either case
    the splitting need not be unique, see \Cref{fig:rectangle_decomposition}.

    \begin{figure}[h]
        \centering
        \begin{tikzpicture}[
            scale=0.9, 
            transform shape,
            point/.style={circle, fill=black, inner sep=2pt},
            hcell/.style={blue, very thick},
            vcell/.style={red, very thick},
            >=Stealth
        ]
            
            \begin{scope}[shift={(0,0)}]
                \draw[thick] (0,0) rectangle (6,3);
                
                \draw[hcell, dashed] (3,0) -- (3,3);
                
                \fill[blue!10] (0,0) rectangle (3,3);
                \fill[cyan!10] (3,0) rectangle (6,3);
                
                \draw[thick] (0,0) rectangle (6,3);
                \draw[hcell] (3,0) -- (3,3);
                
                \node[blue] at (1.5,1.5) {$\RR_A$};
                \node[cyan] at (4.5,1.5) {$\RR_B$};
                
                \node[point] at (0,0) {};
                \node[point] at (3,0) {};
                \node[point] at (6,0) {};
                \node[point] at (0,3) {};
                \node[point] at (3,3) {};
                \node[point] at (6,3) {};
                
                \node at (3,-0.5) {$\RR = \RR_A \COMP^\RECT_H \RR_B$};
                \node[align=center, text width=8cm] at (3,-1.5) 
                    {Alternative horizontal compositions into non-degenerate rectangles};
            \end{scope}
            
            \begin{scope}[shift={(7,0)}]
                \draw[thick] (0,0) rectangle (6,3);
                
                \draw[vcell, dashed] (0,1.5) -- (6,1.5);
                
                \fill[red!10] (0,1.5) rectangle (6,3);
                \fill[magenta!10] (0,0) rectangle (6,1.5);
                
                \draw[thick] (0,0) rectangle (6,3);
                \draw[vcell] (0,1.5) -- (6,1.5);
                
                \node[red] at (3,2.25) {$\RR_{A'}$};
                \node[magenta] at (3,0.75) {$\RR_A$};
                
                \node[point] at (0,0) {};
                \node[point] at (6,0) {};
                \node[point] at (0,1.5) {};
                \node[point] at (6,1.5) {};
                \node[point] at (0,3) {};
                \node[point] at (6,3) {};
                
                \node at (3,-0.5) {$\RR = \RR_A \COMP^\RECT_V \RR_{A'}$};
                \node[align=center, text width=8cm] at (3,-1.5) 
                    {Alternative vertical compositions into non-degenerate rectangles};
            \end{scope}
            
            \begin{scope}[shift={(0,4)}]
                \draw[thick] (0,0) rectangle (6,3);
                
                \draw[hcell, dashed] (2,0) -- (2,3);
                
                \fill[blue!10] (0,0) rectangle (2,3);
                \fill[cyan!10] (2,0) rectangle (6,3);
                
                \draw[thick] (0,0) rectangle (6,3);
                \draw[hcell] (2,0) -- (2,3);
                
                \node[blue] at (1,1.5) {$\acute \RR_A$};
                \node[cyan] at (4,1.5) {$\acute \RR_B$};
                
                \node[point] at (0,0) {};
                \node[point] at (2,0) {};
                \node[point] at (6,0) {};
                \node[point] at (0,3) {};
                \node[point] at (2,3) {};
                \node[point] at (6,3) {};
                
                \node at (3,-0.5) {$\RR = \acute \RR_A \COMP^\RECT_H \acute \RR_B$};
            \end{scope}
            
            \begin{scope}[shift={(7,4)}]
                \draw[thick] (0,0) rectangle (6,3);
                
                \draw[vcell, dashed] (0,2) -- (6,2);
                
                \fill[red!10] (0,2) rectangle (6,3);
                \fill[magenta!10] (0,0) rectangle (6,2);
                
                \draw[thick] (0,0) rectangle (6,3);
                \draw[vcell] (0,2) -- (6,2);
                
                \node[red] at (3,2.5) {$\acute \RR_{A'}$};
                \node[magenta] at (3,1) {$\acute \RR_A$};
                
                \node[point] at (0,0) {};
                \node[point] at (6,0) {};
                \node[point] at (0,2) {};
                \node[point] at (6,2) {};
                \node[point] at (0,3) {};
                \node[point] at (6,3) {};
                
                \node at (3,-0.5) {$R = \acute \RR_A \COMP^\RECT_V \acute \RR_{A'}$};
            \end{scope}
            
        \end{tikzpicture}
        \caption{Decomposition of a non-elementary rectangle $\RR$ into two non-degenerate rectangles.}
        \label{fig:rectangle_decomposition}
    \end{figure}
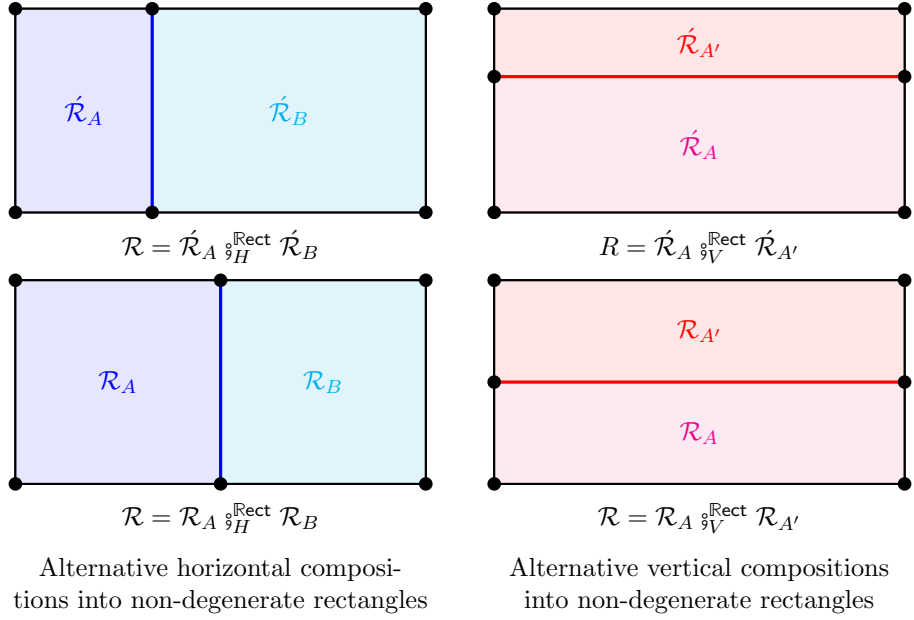

    We then define
    \begin{align*}
        \hat F( \RR ) := \hat F( \RR_A ) \COMP^{\mathbb D}_H \hat F(\RR_B)
    \end{align*}
    or
    \begin{align*}
        \hat F( \RR ) := \hat F( \RR_A ) \COMP^{\mathbb D}_V \hat F(\RR_{A'}).
    \end{align*}

    We need to show that this is well-defined, i.e.
    that if $\RR$ is of size $m\times n$ and $\hat F$ is well-defined
    for all rectangles strictly smaller than $m\times n$,
    then any horizontal or vertical splitting leads to the same value
    $\hat F(\RR)$.
    If we compare two horiztonal splittings, then
    this follows from associativity.
    Analogously for two vertical splittings.

    Let then
    \begin{align*}
        \RR = \RR_W \COMP^{\RECT}_H \RR_E \quad
        \RR = \RR_{S} \COMP^{\RECT}_V \RR_{N},
    \end{align*}
    be two splittings into non-degenerate rectangles.
    Define
    \begin{align*}
        \RR_{SW} := \RR_W \cap \RR_S,\quad
        \RR_{SE} := \RR_E \cap \RR_S,\quad
        \RR_{NW} := \RR_W \cap \RR_N,\quad
        \RR_{NE} := \RR_E \cap \RR_N.
    \end{align*}

    Then, by assumption and the interchange property
    \begin{align*}
        \hat F(\RR_W) \COMP^{\mathbb D}_H \hat F(\RR_E)
        &=
        (\hat F(\RR_{SW}) \COMP^{\mathbb D}_V \hat F(\RR_{NW}))
        \COMP^{\mathbb D}_H
        (\hat F(\RR_{SE}) \COMP^{\mathbb D}_V \hat F(\RR_{NE})) \\
        &=
        (\hat F(\RR_{SW}) \COMP^{\mathbb D}_H \hat F(\RR_{SE}))
        \COMP^{\mathbb D}_V
        (\hat F(\RR_{NW}) \COMP^{\mathbb D}_H \hat F(\RR_{NE})) \\
        &=
        \hat F(\RR_S) \COMP^{\mathbb D}_V \hat F(\RR_N),
    \end{align*}
    as desired.
    This shows that $\hat F$ is well-defined.

    \medskip

    $\hat F$ is a double functor by construction.
    Indeed, let for example $\RR = \RR_A \COMP^\RECT_H \RR_B$.
    Then
    \begin{align*}
        \hat F( \RR ) &= \hat F( \RR_A ) \COMP^{\mathbb D}_H \hat F(R_B).
    \end{align*}

    Finally, any functor coinciding with $\hat F$ on elementary squares
    must coincide with $\hat F$ on all rectangles, this shows uniqueness.

\end{proof}

\subsection{Aggregation as a double functor on a double category of rectangles}
\label{ss:aggregation_as_double_functor}

\begin{quote}
    \itshape
    {
        Many aggregation operations for $2$-parameter data (e.g. images),
        can be interpreted as functors on a double category of rectangles.
    }
\end{quote}

As in \Cref{ss:aggregation_as_functor} we can now aggregate values
attached to elementary rectangles
to values attached to larger rectangles.
In detail,
let $R$ be the double graph of \Cref{thm:free_double_category},
let $\DD$ be any double category
and let $F: R \to \DD$ be a morphism of double graphs.
This means that
\begin{align*}
    \sC_H( F_2( [i,i+1] \times [j,j+1] ) )
    &=
    F^v_1( \{i\} \times [j,j+1] ) \\
    \tC_H( F_2( [i,i+1] \times [j,j+1] ) )
    &=
    F^v_1( \{i+1\} \times [j,j+1] ) \\
    \sC_V( F_2( [i,i+1] \times [j,j+1] ) )
    &=
    F^h_1( [i,i+1] \times \{j\} ) \\
    \tC_V( F_2( [i,i+1] \times [j,j+1] ) )
    &=
    F^h_1( [i,i+1] \times \{j+1\} ) \\
    \sC_v( F^v_1( \{i\} \times [j,j+1] )
    &=
    F_0( \{i\} \times \{j\} ) \\
    \tC_v( F^v_1( \{i\} \times [j,j+1] )
    &=
    F_0( \{i\} \times \{j+1\} ) \\
    \sC_h( F^h_1( [i,i+1] \times \{j\} )
    &=
    F_0( \{i\} \times \{j\} ) \\
    \tC_h( F^h_1( [i,i+1] \times \{j\} )
    &=
    F_0( \{i+1\} \times \{j\} ).
\end{align*}

If $\DCD$ has only one object,
as is the case for the ``delooping'' of a crossed module \Cref{ex:cm_to_double},
then the last four conditions are void.

Then, by the universal property of $\RECT$,
there exists a unique double functor
\begin{align*}
    \hat F: \RECT \to \DD
\end{align*}
extending $F$.
In particular, we have aggregated values
\begin{align*}
   \hat F( [s_1,t_1] \times [s_2,t_2] )
\end{align*}
attached to all rectangles.

\begin{example}
    \label{ex:aggregation}
    Let $\DD$ be the double category of 
    \Cref{ex:abelian_group}, with $A$ the abelian group $(\R,+)$.

    The example 
    \Cref{ex:double_sum} is obtained by lifting the morphism of double graphs
    \begin{align*}
        F_0(\k)                        &:= *, \\
        F^h_1([s_1,t_1]\times \{x_2\}) &:= *, \\
        F^v_1(\{x_1\}\times [s_2,t_2]) &:= *, \\
        F_2([i,i+1]\times[j,j+1])      &:= x_{i,j}.
    \end{align*}
\end{example}

\begin{example}
    We use the delooping of the crossed module \Cref{ex:crossed_modules:general_linear},
    with $n=2, p=1, q=3$.

    Let $x_\k \in \R^3$, $\k \in \Z^2$, be spatial data
    (an ``RGB image'').
    We want to attach values in $G$ to edges
    and values in $H$ to faces that 
    are compatible in the sense of \Cref{ex:cm_to_double}.

    For concreteness,
    fix $A_1,A_2,A_3 \in \R^{2\times 2}$,
    $Q_1, Q_2, Q_3 \in \R^{3\times 3}$,
    $s_1,s_2,s_3 \in \R$
    and let%
    \footnote{
        The expressions here are arbitrary.
        We just have to make sure that certain submatrices are invertible,
        which we ensure by taking a matrix exponential.
        }
    \begin{align*}
        \eta(z,\bar z)
        &:= 
        (
        \begin{bmatrix}
            \exp\left( \sum_{k=1}^3 A_k \Delta z^{(k)} \right) & 0 \\
            \begin{matrix}
               \sin(\Delta z^{(1)}) & \cos(\Delta z^{(3)})
            \end{matrix}
            & \exp\left( \sum_{k=1}^3 s_k \Delta z^{(k)} \right)
        \end{bmatrix},\\
        &\qquad\quad
        \begin{bmatrix}
            \exp\left( \sum_{k=1}^3 A_k \Delta z^{(k)} \right) &
            \begin{matrix}
                0 & \Delta z^{(1)} & 0 \\
                \Delta z^{(3)} & 0 & \Delta z^{(2)} \\
            \end{matrix} \\
            0 & \exp\left( \sum_{k=1}^3 Q_k \Delta z^{(k)} \right)
        \end{bmatrix} ),
    \end{align*}
    where $\Delta z = z - \bar z$.
    Then, set
    \begin{align*}
        f_1^h( [i,i+1] \times \{j\} )
        :=
        \eta( x_{i+1,j}, x_{i,j} ),
        f_1^v( \{i\} \times [j,j+1] )
        &=
        \eta( x_{i,j+1}, x_{i,j} ).
    \end{align*}
    The elements
    \begin{align*}
        f_2( [i,i+1] \times [j,j+1] ),
    \end{align*}
    must be compatible in the sense that
    \begin{align*}
        \feedbackGG( f_2( [i,i+1] \times [j,j+1] ) )
        &=
        f^h_1( [i,i+1] \times \{j\} )
        f^v_1( \{i+1\} \times [j,j+1] ) \\
        &\qquad
        \times
        f^h_1( [i,i+1] \times \{j+1\} )^{-1}
        f^v_1( \{i\} \times [j,j+1] )^{-1}.
    \end{align*}
    As a consequence,
    \begin{align*}
        f_2( [i,i+1] \times [j,j+1] ) )
        =
        \begin{bmatrix}
            P - \id_{2\times 2} & B \\ 
            R & N,
        \end{bmatrix}
    \end{align*}
    where $P,B,R$ are explicit expressions
    in the boundary elements,
    and $N$ can be any $1\times 3$ matrix,
    and for concreteness we set
    \begin{align*}
        N &=
       \begin{bmatrix}
       \sum_{k=1}^3 (z^{(k)}_{i+1,j+1}-z^{(k)}_{i,j})^2
        &
        0
        &
        \sum_{k=1}^3
            (z^{(k)}_{i+1,j+1}-z^{(k)}_{i+1,j})
            (z^{(k)}_{i+1,j+1}-z^{(k)}_{i,j+1})
        \end{bmatrix}.
    \end{align*}
\end{example}

%
%

\subsection{Parallel scan}
\label{ss:parallel_scan_2}

\begin{quote}
    \itshape
    {
        The parallel scan algorithm discussed in Section~\ref{ss:parallel_scan_1}
        for sequential data can be applied to 
        the double functors of \Cref{ss:aggregation_as_double_functor}
        ``row by row'' and then ``column by column''.
    }
\end{quote}


Given a double functor $\hat{F}: \RECT \rightarrow \DCD$,
we can compute aggregated values
$\hat{F}([0,i] \times [0,j])$, $i \in \{0,1,\ldots,m-1\}$, $j \in \{0,1,\ldots,n-1\}$,
by applying Blelloch's parallel scan algorithm in two phases.

\begin{enumerate}

\item
\textbf{Row-wise scan}: For each row $j \in \{1, \ldots, n\}$,
apply the parallel scan algorithm to compute
\begin{align*}
    g(i,j)
    &:= \hat{F}_2([0,i] \times [j-1,j]) \\
    &= \hat{F}_2([0,i-1] \times [j-1,j]) \COMP_H \hat{F}_2([i-1,i] \times [j-1,j]).
\end{align*}
All rows can be computed in parallel.

\item
\textbf{Column-wise scan}:
Then, for each column $i \in \{1, \ldots, m\}$,
\begin{align*}
    \hat{F}([0,i] \times [0,j])
    =
    \hat{F}([0,i] \times [0,j-1]) \COMP_V g(i,j),
\end{align*}
can also be computed using the parallel scan algorithm.
Againn, all columns can be computed in parallel.

\end{enumerate}

The total time-complexity with $p$ parallel processors
then is $\O((m \cdot n)/p + \log p)$.

\section{Outlook}
\label{sec:conclusion}


\begin{itemize}

    \item 
    Data on irregular geometries: rooted trees, posets, graphs.
    Here, more flexible categorial structures are needed.
    For example, we hypothesize that aggregation on rooted trees would benefit
    from a description in terms of hypergraph categories
    (\cite{baez2018compositional,fong2019hypergraph}).

    \item
    Gauge transformations play an important
    role in path holonomy (which can be seen as a continuous version
    of aggregation; now on smooth enough curves)
    and surface holonomy (which can be seen as a continuous version
    of aggregation; now on smooth enough surfaces);
    see for example \cite{baez2004higher,schreiber2009parallel,schreiber2011smooth}.
    They are also important in the context of
    lattice gauge theories (\cite{parzygnat2019two})
    and tensor networks (\cite{tindall2023gauging}).
    We are unsure what role they will play in the context of data aggregation.

    \item
    Our main example of a double category is 
    the delooping of a crossed module.
    This in fact yields a double groupoid: every morphism is invertible.
    Non-trivial
    (in particular: non-abelian, and the feedback map neither injective nor trivial)
    examples of crossed modules of monoids 
    would provide us with non-trivial, concrete examples of double categories
    which are not double groupoids.
    This could prove very beneficial for applications,
    since a double groupoid imposes a lot of structure
    on the aggregation. In particular, every aggregation
    step can be ``undone'', which leads to a large group of symmetries
    that are not ``seen'' by the aggregation.
    Some work on crossed modules of monoids
    exists (\cite{patchkoria1998crossed,bohm2020crossed,temel2022some,pirashvili2024internal}),
    but all the examples given have either an injective
    feedback map or are abelian.
    
\end{itemize}

\printbibliography

\appendix

\section{Appendix}
\label{sec:appendix}

\begin{definition}
A \DEF{double category} $\DCC$ consists of
\begin{itemize}
    \item a set of \DEF{objects} or \DEF{0-cells} $\DCC_0$,
    \item a set of \DEF{vertical morphisms} or \DEF{vertical 1-cells} $\DCC^v_1$,
    \item a set of \DEF{horizontal morphisms} or \DEF{horizontal 1-cells} $\DCC^h_1$,
    \item a set of \DEF{faces} or \DEF{2-cells} $\DCC_2$,
\end{itemize}
together with maps
\begin{center}
    \begin{tikzcd}[row sep=2em, column sep=2em]
        \DCC_2 \arrow[dd, "\partial_V^+" description, shift left=3, near end] \arrow[dd, "\partial_V^-" description, shift right=3, near end] \arrow[rr, "\partial_H^+" description, shift left=3, near end] \arrow[rr, "\partial_H^-" description, shift right=3, near end] &  & \DCC_1^v \arrow[dd, "\partial_v^-" description, shift right=3, near end] \arrow[dd, "\partial_v^+" description, shift left=3, near end] \arrow[ll, "\i_{H}" description, near end] \\
        & & \\
        \DCC_1^h \arrow[rr, "\partial_h^-" description, shift left=3, near end] \arrow[rr, "\partial_h^+" description, shift right=3, near end] \arrow[uu, "\i_{V}" description, near end]                                                                               &  & \DCC_0 \arrow[ll, "\i_{h}" description, near end] \arrow[uu, "\i_{v}" description, near end]                                                                              
    \end{tikzcd}
    \end{center}
and ``composition'' maps
\begin{align*}
    \COMP_v: \DCC^v_1 \times_{\DCC_0} \DCC^v_1
    &:= \{ (f,g) \in \DCC^v_1 \times \DCC^v_1 \mid \partial_v^+(f) = \partial_v^-(g) \}
    \to \DCC^v_1, \\
    \COMP_h: \DCC^h_1 \times_{\DCC_0} \DCC^h_1 &:= \{ (f,g) \in \DCC^h_1 \times \DCC^h_1 \mid \partial_h^+(f) = \partial_h^-(g) \}
    \to \DCC^h_1, \\
    \COMP_V: \DCC_2 \times_{\DCC^h_1} \DCC_2 &:= \{ (\alpha,\beta) \in \DCC_2 \times \DCC_2 \mid \partial_H^+(\alpha) = \partial_H^-(\beta) \}
    \to \DCC_2, \\
    \COMP_H: \DCC_2 \times_{\DCC^v_1} \DCC_2 &:= \{ (\alpha,\beta) \in \DCC_2 \times \DCC_2 \mid \partial_V^+(\alpha) = \partial_V^-(\beta) \}
    \to \DCC_2,
\end{align*}
that satisfy the following axioms:
\begin{itemize}
\item
    $(\categoryConstellation{\DCC_0}{\DCC^h_1}{\partial^-_h}{\partial^+_h}{\i_{h}}, \COMP_{h} )$
forms a category (\textbf{horizontal $1$-cells}),

\item
    $( \categoryConstellation{\DCC_0}{\DCC^v_1}{\partial^-_v}{\partial^+_v}{\i_{v}}, \COMP_{v} )$
forms a category (\textbf{vertical $1$-cells}),

\item 
    $( \inlineCategoryConstellation{\DCC^v_1}{\DCC_2}{\partial^-_H}{\partial^+_H}{\i_{H}}, \COMP_{H} )$,
forms a category (\textbf{horizontal composition of $2$-cells}),

\item
    $( \inlineCategoryConstellation{\DCC^h_1}{\DCC_2}{\partial^-_V}{\partial^+_V}{\i_{V}}, \COMP_{V} )$
forms a category (\textbf{vertical composition of $2$-cells}),


    \item ``the interchange law'' holds, i.e. for composable $2$-cells
    \begin{align*}
        ( \alpha \COMP_{H} \beta ) \COMP_{V} ( \alpha' \COMP_{H} \beta' )
        =
        ( \alpha \COMP_{V} \alpha' ) \COMP_{H} ( \beta \COMP_{V} \beta' ).
    \end{align*}
    
    \item ``the four corners of a $2$-cell are well-defined'', i.e.
    \begin{align*}
        \partial_V^- \COMP \partial_h^- &= \partial_H^- \COMP \partial_v^- & \partial_V^+ \COMP \partial_h^- &= \partial_H^- \COMP \partial_v^+  \\
        \partial_V^- \COMP \partial_h^+ &= \partial_H^+ \COMP \partial_v^- & \partial_V^+ \COMP \partial_h^+ &= \partial_H^+ \COMP \partial_v^+,
    \end{align*}

    \item
    ``the boundaries are morphisms'', i.e.
    \begin{align*}
       (\partial^+_v,\partial^+_V):
        ( \inlineCategoryConstellation{\DCC^v_1}{\DCC_2}{\partial^-_H}{\partial^+_H}{\i_H},\COMP_H )
            \to
        ( \inlineCategoryConstellation{\DCC_0}{\DCC^h_1}{\partial^-_h}{\partial^+_h}{\i_h},\COMP_h ) \\
       (\partial^+_h,\partial^+_H):
        ( \inlineCategoryConstellation{\DCC^h_1}{\DCC_2}{\partial^-_V}{\partial^+_V}{i_V},\COMP_V )
            \to
        ( \inlineCategoryConstellation{\DCC_0}{\DCC^v_1}{\partial^-_v}{\partial^+_v}{i_0},\COMP_v ),
    \end{align*}
    are functors.
    Analogously for $(\partial^-_v,\partial^-_V)$ and $(\partial^-_h,\partial^-_H)$.%
    \footnote{Recall \Cref{not:functor}:
    $\partial^+_v$ is the object-part of a functor
    and $\partial^+_V$ is the morphism-part of a functor, etc.}


    \item ``identities are compatible'', i.e.
    \begin{align*}
       \i_V( \i_h(x) ) = \i_H( \i_v(x) ) \qquad \forall x \in \DCC_0.
    \end{align*}

\end{itemize}

A double category $\DCC$ is called \DEF{edge-symmetric} (\cite{brown1999double}),
if the two categories of $1$-cells are the same,
i.e.  $\DCC^v_1 = \DCC^h_1 = \DCC_1$,
$\COMP_h = \COMP_v =: \COMP$,
and
$\partial^\pm_v = \partial^\pm_h =: \partial^\pm$.

\bigskip

A \DEF{double functor} $F: \DCC \to \DCD$ between double categories
is given by four maps $F_i: \DCC_i \to \DCD_i$, $i=0,2$, $F_1^h: \DCC_1^h \to \DCD_1^h$, $F_1^v: \DCC_1^v \to \DCD_1^v$,
which commute with all the structure maps of double categories.

\end{definition}

\bigskip

\begin{definition}
    A \DEF{double graph} $\DGG$
    is a commuting diagram in $\SET$ of the form
    \begin{center}
    \begin{tikzcd}
        \DGG_2
        \arrow[dd, "\partial_V^+" description, shift left=3]
        \arrow[dd, "\partial_V^-" description, shift right=3]
        \arrow[rr, "\partial_H^+" description, shift left=3]
        \arrow[rr, "\partial_H^-" description, shift right=3] &  & \DGG^v_1
        \arrow[dd, "\partial_v^-" description, shift right=2]
        \arrow[dd, "\partial_v^+" description, shift left=2] \\
        & & \\
        \DGG^h_1
        \arrow[rr, "\partial_h^-" description, shift left=3]
        \arrow[rr, "\partial_h^+" description, shift right=3]
         &  & \DGG_0 \end{tikzcd}
    \end{center}

    A \DEF{morphism of double graphs} $F: \DGG \to \DGH$
    is given by four maps $F_i: \DGG_i \to \DGH_i$, $i=0,2$, $F_1^h: \DGG_1^h \to \DGH_1^h$, $F_1^v: \DGG_1^v \to \DGH_1^v$,
    which commute with all the structure maps of double graphs.

\end{definition}

\bigskip
\bigskip

\begin{example}
    \label{ex:full_cm_to_double}

    Continuing \Cref{ex:cm_to_double}.
    First, the composition is well-defined, i.e.
    the composed $2$-cells satisfy the boundary condition.
    First, we check the horizontal composition
    (recall that for the composition to be well-defined, $x_e = y_w$):
	\begin{align*}
        \feedbackGG( \actionGG_{x_s}(Y) X ) x_w (x_n y_n)
        &=
        \feedbackGG( \actionGG_{x_s}(Y) ) \feedbackGG( X ) x_w x_n y_n
        =
        \feedbackGG( \actionGG_{x_s}(Y) ) x_s x_e y_n \\
        &=
        x_s \feedbackGG( Y ) x_e y_n
        =
        x_s \feedbackGG( Y ) y_w y_n
        =
        x_s y_s y_e,
	\end{align*}
    as desired.
    Regarding the vertical composition, we have
    (recall that for the composition to be well-defined, $x_n = x'_s$):
    \begin{align*}
        \feedbackGG( X \actionGG_{x_w}(X') ) (x_w x'_w) x'_n
        &=
        \feedbackGG( X ) \feedbackGG( \actionGG_{x_w}(X') ) x_w x'_w x'_n \\
        &=
        \feedbackGG( X ) x_w \feedbackGG( X' ) x'_w x'_n \\
        &=
        \feedbackGG( X ) x_w x'_s x'_e \\
        &=
        \feedbackGG( X ) x_w x_n x'_e \\
        &=
        x_s x_e x'_e,
    \end{align*}
    as desired.

    Associativity of the composition follows from a straightforward calculation.

    We verify the interchange law.

    \begin{center}
    \begin{tikzpicture}[
        scale=0.9, 
        transform shape, 
        cell/.style={
            rectangle,
            draw,
            minimum width=2cm,
            minimum height=2cm,
        },
        boundary/.style={
            draw=green!70!black,
            rounded corners,
            thick,
            inner sep=20pt
        }
    ]
        \node[cell] (Xp) {$X'$};
        \node[cell, right=2cm of Xp] (Yp) {$Y'$};
        \node[font=\Large] at ($(Xp.east)!0.5!(Yp.west)$) {$\COMP_H$};
        
        \node[cell, below=2cm of Xp] (X) {$X$};
        \node[cell, right=2cm of X] (Y) {$Y$};
        \node[font=\Large] at ($(X.east)!0.5!(Y.west)$) {$\COMP_H$};
        
        \node[above] at (Xp.north) {$x'_n$};
        \node[right] at (Xp.east) {$x'_e$};
        \node[below] at (Xp.south) {$x'_s$};
        \node[left] at (Xp.west) {$x'_w$};
        
        \node[above] at (Yp.north) {$y'_n$};
        \node[right] at (Yp.east) {$y'_e$};
        \node[below] at (Yp.south) {$y'_s$};
        \node[left] at (Yp.west) {$y'_w$};
        
        \node[above] at (X.north) {$x_n$};
        \node[right] at (X.east) {$x_e$};
        \node[below] at (X.south) {$x_s$};
        \node[left] at (X.west) {$x_w$};
        
        \node[above] at (Y.north) {$y_n$};
        \node[right] at (Y.east) {$y_e$};
        \node[below] at (Y.south) {$y_s$};
        \node[left] at (Y.west) {$y_w$};
        
        \node[boundary, fit=(Xp) (Yp)] (top_boundary) {};
        \node[boundary, fit=(X) (Y)] (bottom_boundary) {};
        
        \node[font=\Large] (compv) at ($(top_boundary.south)!0.5!(bottom_boundary.north)$) {$\COMP_V$};
        
        \coordinate (top_center) at ($(Xp)!0.5!(Yp)$);
        \coordinate (bottom_center) at ($(X)!0.5!(Y)$);
        
        \node[font=\Large] at ($(compv)+(4.5cm,0)$) {$=$};
        
        \node[cell] (A) at ($(top_center)+(6cm,0)$) {$\actionGG_{x'_s}(Y') X'$};
        
        \node[cell] (B) at ($(bottom_center)+(6cm,0)$) {$\actionGG_{x_s}(Y) X$};
        
        \node[font=\Large] (compv2) at ($(A.south)!0.5!(B.north)$) {$\COMP_V$};
        
        \node[above] at (A.north) {$x'_n y'_n$};
        \node[below] at (A.south) {$x'_s y'_s$};
        \node[left] at (A.west) {$x'_w$};
        \node[right=0.3cm] at (A.east) {$y'_e$};
        
        \node[above] at (B.north) {$x_n y_n$};
        \node[below] at (B.south) {$x_s y_s$};
        \node[left] at (B.west) {$x_w$};
        \node[right=0.3cm] at (B.east) {$y_e$};
        
        \node[font=\Large] at ($(compv2)+(2cm,0)$) {$=$};
        
        \coordinate (center_both) at ($(top_center)!0.5!(bottom_center)$);
        \node[cell, minimum width=2.5cm, minimum height=2.5cm] (C) at ($(center_both)+(10.5cm,0)$) {$A$};
        
        \node[above] at (C.north) {$x'_n y'_n$};
        \node[below] at (C.south) {$x_s y_s$};
        \node[left] at (C.west) {$x_w x_w'$};
        \node[right=0.3cm] at (C.east) {$y_e y_e'$};
    \end{tikzpicture}
    \end{center}

\begin{center}
\begin{tikzpicture}[
    scale=.9, 
    transform shape, 
    cell/.style={
        rectangle,
        draw,
        minimum width=2cm,
        minimum height=2cm,
    },
    boundary/.style={
        draw=green!70!black,
        rounded corners,
        thick,
        inner sep=18pt
    }
]
    \node[cell] (Xp) {$X'$};
    \node[cell, right=2cm of Xp] (Yp) {$Y'$};
    
    \node[cell, below=2cm of Xp] (X) {$X$};
    \node[cell, right=2cm of X] (Y) {$Y$};
    
    \node[font=\Large] at ($(Xp.south)!0.5!(X.north)$) {$\COMP_V$};
    \node[font=\Large] at ($(Yp.south)!0.5!(Y.north)$) {$\COMP_V$};
    
    \node[above] at (Xp.north) {$x'_n$};
    \node[right] at (Xp.east) {$x'_e$};
    \node[below] at (Xp.south) {$x'_s$};
    \node[left] at (Xp.west) {$x'_w$};
    
    \node[above] at (Yp.north) {$y'_n$};
    \node[right] at (Yp.east) {$y'_e$};
    \node[below] at (Yp.south) {$y'_s$};
    \node[left] at (Yp.west) {$y'_w$};
    
    \node[above] at (X.north) {$x_n$};
    \node[right] at (X.east) {$x_e$};
    \node[below] at (X.south) {$x_s$};
    \node[left] at (X.west) {$x_w$};
    
    \node[above] at (Y.north) {$y_n$};
    \node[right] at (Y.east) {$y_e$};
    \node[below] at (Y.south) {$y_s$};
    \node[left] at (Y.west) {$y_w$};
    
    \node[boundary, fit=(Xp) (X)] (left_boundary) {};
    \node[boundary, fit=(Yp) (Y)] (right_boundary) {};
    
    \node[font=\Large] (comph) at ($(left_boundary.east)!0.5!(right_boundary.west)$) {$\COMP_H$};
    
    \coordinate (left_center) at ($(Xp)!0.5!(X)$);
    \coordinate (right_center) at ($(Yp)!0.5!(Y)$);
    
    \node[font=\Large] (firstequal) at ($(comph)+(4.5cm,0)$) {$=$};
    
    \node[cell] (D) at ($(left_center)+(9cm,0)$) {$X \actionGG_{x_w}(X')$};
    
    \node[cell] (E) at ($(right_center)+(10.5cm,0)$) {$Y \actionGG_{y_w}(Y')$};
    
    \node[font=\Large] (comph2) at ($(D.east)!0.5!(E.west)$) {$\COMP_H$};
    
    \node[above] at (D.north) {$x'_n$};
    \node[below] at (D.south) {$x_s$};
    \node[left] at (D.west) {$x_w x_w'$};
    \node[right] at (D.east) {$x_e x'_e$};
    
    \node[above] at (E.north) {$y'_n$};
    \node[below] at (E.south) {$y_s$};
    \node[left] at (E.west) {$y_w y'_w$};
    \node[right] at (E.east) {$y_e y'_e$};
    
    \node[font=\Large] (secondequal) at ($(firstequal)+(0cm,-6cm)$) {$=$};
    
    \node[cell, minimum width=2.5cm, minimum height=2.5cm] (C) at ($(secondequal)+(3cm,0)$) {$B$};
    
    \node[above] at (C.north) {$x'_n y'_n$};
    \node[below] at (C.south) {$x_s y_s$};
    \node[left] at (C.west) {$x_w x_w'$};
    \node[right=0.3cm] at (C.east) {$y_e y_e'$};
\end{tikzpicture}
\end{center}

Finally, we note
\begin{align*}
    \actionGG_{x_s x_e}(Y') X
    =
    \actionGG_{\tau(X) x_w x_n}(Y') X
    =
    \actionGG_{\tau(X)} ( \actionGG_{x_w x_n}(Y') ) X
    =
    X \actionGG_{x_w x_n}(Y'),
\end{align*}
and hence
\begin{align*}
    A
    &= 
    \actionGG_{x_s}(Y) X \actionGG_{x_w}( \actionGG_{x'_s}(Y') X' )
    =
    \actionGG_{x_s}(Y) X \actionGG_{x_w x_n}(Y') \actionGG_{x_w}( X' ) \\
    &=
    \actionGG_{x_s}( Y \actionGG_{y_w}(Y') ) X \actionGG_{x_w}(X')
    =
    B.
\end{align*}

\end{example}

\section*{Acknowledgments}

The author thanks 
Ilya Chevyrev,
Kurusch Ebrahimi-Fard,
Darrick Lee
and
Nikolas Tapia
for fruitful discussions on the categorial view
on two-parameter data.
Part of this work was carried out during the
\emph{2024 Lie-Stormer Colloquium Foundational and computational aspects of symmetry}
(\url{https://sites.google.com/view/lscolloq2024/home})
and the
\emph{2024 TMS Colloquium Multi-Parameter Signatures}
(\url{https://sites.google.com/view/2dsignatures/home}).

\end{document}